\documentclass[12pt]{article}

\usepackage{amsmath}
\usepackage{amssymb}
\usepackage{amscd}
\usepackage{bm}
\usepackage{mathrsfs}
\usepackage[all]{xy}

\usepackage{graphicx}

\usepackage{bbm}
\usepackage{latexsym}
\usepackage[all]{xy}
\usepackage{makeidx}
\usepackage{tocloft}
\usepackage{fancyhdr}

\usepackage{ifthen}

\newcommand{\lsection}[2][""]{%
    \ifthenelse{\equal{#1}{""}}{%
        \section{#2}
    }{%
        \renewcommand{\sectionmark}[1]{\markright{\thesection.\ \MakeUppercase{#1}}}
        \section{#2}
        \renewcommand{\sectionmark}[1]{\markright{\thesection.\ \MakeUppercase{##1}}}
    }
}

\newcommand{\lchapter}[2][""]{%
    \ifthenelse{\equal{#1}{""}}{%
        \chapter{#2}
    }{%
        \renewcommand{\chaptermark}[1]{\markboth{\MakeUppercase{\chaptername\ \thechapter.\ #1}}{}}
        \chapter{#2}
        \renewcommand{\chaptermark}[1]{\markboth{\MakeUppercase{\chaptername\ \thechapter.\ ##1}}{}}
    }
}

\def\Z{\mathbb Z}

\def\R{\mathbb R}
\def\C{\mathbb C}

\def\iff{if and only if}
\def\mfd{manifold}
\def\fcn{function}
\def\str{structure}

\def\h{holomorphic}

\def\r{respectively}
\def\st{such that}
\def\(#1_#2){(#1_1,#1_2,\dots,#1_#2)}
\def\p #1_#2{#1_1#1_2\dots#1_#2}
\def\s#1_#2{#1_1+#1_2+\dots+#1_#2}
\def\wrt{with respect to}

\def\iso{isomorphism}
\def\ra{\rightarrow}
\def\lra{\longrightarrow}
\def\lla{\longleftarrow}

\def\hra{\hookrightarrow}
\def\op{\operatorname}

\def\vb{vector bundle}

\def\ssl{\smallsmile}
\def\bp{\bar\partial}

\def\ssm{\hspace{-.5mm}\smallsetminus\hspace{-.5mm}}

\def\nbd{neighborhood}

\def\scB{{\mathscr B}}

\def\scT{{\mathscr T}}

\def\scZ{{\mathscr Z}}

\def\scF{\mathscr F}

\def\scA{\mathscr A}

\def\V{\mathcal V}
\def\W{\mathcal W}

\def\scO{\mathscr O}

\def\scS{\mathscr S}

\def\scE{\mathscr E}

\def\scK{\mathscr K}

\def\scD{\mathscr D}

\def\a{\alpha}
\def\b{\beta }

\def\o{\omega}

\def\t{\theta}

\def\vP{\varPsi}
\def\vG{\varGamma}

\def\ve{\varepsilon}
\def\vt{\vartheta}
\def\vPhi{\varPhi}

\def\sV{\mathsf{V}}
\def\sW{\mathsf{W}}

\def\vv {\vskip.2cm}

\DeclareMathOperator{\sgn}{sgn}

\newtheorem{theorem}{Theorem}[section]
\newtheorem{lemma}[theorem]{Lemma}
\newtheorem{corollary}[theorem]{Corollary}
\newtheorem{proposition}[theorem]{Proposition}

\newtheorem{exa}[theorem]{Example}
\newenvironment{example}{\begin{exa} \em}{\end{exa}}
\newtheorem{exas}[theorem]{Examples}

\newtheorem{prope}[theorem]{Property}

\newtheorem{defini}[theorem]{Definition}
\newenvironment{definition}{\begin{defini} \em}{\end{defini}}

\newtheorem{rema}[theorem]{Remark}
\newenvironment{remark}{\begin{rema} \em}{\end{rema}}

\newenvironment{equationth}{\stepcounter{theorem}\begin{equation}}{\end{equation}}

\newenvironment{proof}{{\noindent \sc Proof: } }{\mbox{ }\hfill$\Box$  
                        \vspace{1.5ex} \par}

\title {\bf  \Large {Relative Dolbeault cohomology}}

\author{Tatsuo Suwa\thanks{Supported by  JSPS grants 
24540060 and 16K05116.}
}

\date{}

\begin{document}

\pagestyle{plain}


\maketitle

\begin{center} {\it In memory of Pierre Dolbeault}
\end{center}


\noindent
{\bf Abstract}

We review the notion of relative Dolbeault cohomology and prove that it is canonically
isomorphic with the local (relative) cohomology of A. Grothendieck and M. Sato with coefficients in the sheaf of \h\ forms. We deal with this cohomology from two viewpoints. One is  the \v{C}ech theoretical approach, which is convenient to
define such operations as the cup product and integration and leads to the study of local duality.
Along the way we 
also establish some notable canonical
\iso s among various cohomologies.
The other is  to regard it as the cohomology of a certain complex,
which is interpreted as a notion dual to the mapping cone in the theory of derived categories.
This approach shows that the cohomology goes well with  derived functors.
We also give some examples and indicate  applications, including  simple  explicit expressions of 
Sato hyperfunctions, fundamental operations on them and related local duality theorems.
\bigskip


\vv


\lsection{Introduction}

In \cite{Su11} we discussed the cohomology theory of sheaf complexes for open embeddings of topological
spaces and gave some general ways of representing the relative  cohomology of a sheaf in terms of 
a soft or fine resolution of the sheaf. In this paper we apply the theory to the case of Dolbeault complex. 
This naturally leads to the notion of the relative Dolbeault cohomology of a complex \mfd. As is explained in \cite{Su11}, there are two ways to approach this cohomology.
One is to define it as a special case of \v{C}ech-Dolbeault cohomology. This viewpoint goes well with such operations as the cup product and the integration, which enable us to deal with the local duality problem.
The integration theory here is a descendent of the one on the \v{C}ech-de Rham cohomology, which is defined using  honeycomb systems. The other is to see it as the cohomology of a certain complex
called co-mapping cone, a notion dual to the mapping cone in the theory of derived categories.
From
this viewpoint we see that the cohomology goes well with derived functors.
 It is readily generalized to the cohomology of \h\ maps between complex
\mfd s. In any case we have the relative Dolbeault theorem which says that the relative Dolbeault cohomology is canonically isomorphic with the local  (relative) cohomology of A. Grothendieck and M. Sato (cf. \cite{Ha}, \cite{Sa}) with coefficients in the sheaf of \h\ forms (cf. Theorems ~\ref{thdrel} and \ref{thdrel2}). We also present some canonical \iso s that appear along the way (Theorem \ref{thsummary}) and give some examples and applications.

\

The paper is organized as follows.  In Section \ref{secDoe}, we introduce the Dolbeault cohomology for open embeddings of
complex \mfd s and state the aforementioned relative Dolbeault theorem. We also give  generalizations of them to the case of
\h\ maps. We recall, in Section \ref{secCD}, the \v{C}ech-Dolbeault cohomology and some related canonical \iso s. 
In Section \ref{secrDol} we review the relative Dolbeault cohomology from the \v{C}ech theoretical viewpoint and indicate an alternative proof of the relative Dolbeault theorem.

There are certain cases where there is a significant relation between the de~Rham and Dolbeault cohomologies, which
are taken up in Section \ref{secdRD}.
In Section \ref{seccupint}, we discuss the cup product and the integration. As mentioned above, the integration theory is a descendant of the one 
on the \v{C}ech-de~Rham cohomology, which we briefly recall.
In Section \ref{secld}, we discuss global and local dualities. In the global case where the \mfd\ is compact, we have the Kodaira-Serre duality. In the local case we have the duality morphism, which we call the $\bp$-Alexander morphism.
We prove an exact sequence and a commutative diagram giving relation between global and local dualities (Theorems~\ref{thexactloc} and \ref{thlongdual}). We then recall the theory of Fr\'echet-Schwartz and dual Fr\'echet-Schwartz
spaces and state a theorem where we have the local duality (Theorem \ref{pbalex}).

Finally we give in Section \ref{secapp}, some examples and applications. The correspondence of the Bochner-Martinelli form and the higher dimensional Cauchy form in the \iso\ of the Dolbeault and  \v{C}ech cohomologies
is rather well-known (cf. \cite{GH}, \cite{H2}). Here we give the canonical correspondence of them together with
integrations  in our
context (Theorems~\ref{BMCauchy} and \ref{thcdcorrint}).  We also present the local duality theorem of A. Martineau in our framework (Theorem \ref{thMH}) and, as a special case, describe the 
local residue pairing. These are closely related to the Sato hyper\fcn\ theory. In fact,  it is one of the major topics to which the relative Dolbeault theory can be applied. This application to hyper\fcn s 
is discussed in detail in \cite{HIS}. Here we take up some of the essences. As another important applications, there is the localization theory of Atiyah classes, including
the theory of analytic Thom classes (cf. \cite{ABST}, \cite{ASTT}, \cite{Su8}).

\

The author would like to thank Naofumi Honda for stimulating discussions and 
valuable comments, in particular on some materials in Section \ref{secld}.
Thanks are also due to 
Takeshi Izawa and  Toru Ohmoto
for inspiring conversations.
\lsection{Dolbeault cohomology of  open embeddings}\label{secDoe}

In this section we recall the contents of  \cite[Section 2]{Su11}  specializing them to our setting.

In the sequel, by a sheaf we mean a sheaf with at least the \str\ of  Abelian groups.
For a sheaf $\scS$ on a topological space $X$ and an open set $V$ in $X$, we denote by $\scS(V)$ the group of
sections of $\scS$ on $V$. Also for an open subset $V'$ of $V$, we denote by $\scS(V,V')$ the sections on $V$
that vanish on $V'$.


\subsection{Cohomology via flabby resolutions}
As  reference cohomology theory, we adopt the one via flabby resolutions.
Recall that a sheaf $\scF$ is {\em flabby} if the restriction 
$\scF(X)\ra\scF(V)$ is surjective for any open set $V$ in $X$.
Recall also
that every sheaf admits a flabby resolution. Let $\scS$ be a sheaf on $X$ and
\[
0\lra \scS\lra \scF^{0}\overset{d}\lra\cdots\overset{d}\lra \scF^{q}\overset{d}\lra\cdots
\]
 a flabby resolution of $\scS$. For an open set $X'$ in $X$, the $q$-th cohomology $H^{q}(X,X';\scS)$ of $(X,X')$ with coefficients in $\scS$ is the $q$-th cohomology
of the complex $(\scF^{\bullet}(X,X'),d)$. 
Note that it
 is determined uniquely modulo canonical \iso s, independently of the flabby
resolution. We denote $H^{q}(X,\emptyset;\scS)$ by $H^{q}(X;\scS)$. 
We have $H^{0}(X,X';\scS)=\scS(X,X')$. 
Setting $S=X\ssm X'$, it will also be denoted by $H^{q}_{S}(X;\scS)$. 
This cohomology in the first expression is referred to as the {\em relative cohomology} of $\scS$ on $(X,X')$
and in the
second expression the {\em local cohomology} of 
$\scS$ on $X$ with support in $S$
(cf. \cite{Ha}, \cite{Sa}). 

\subsection{Dolbeault cohomology}\label{ssDol}
Let $X$ be a complex \mfd\ of dimension $n$. We always assume that it has a countable basis so that it is 
paracompact and has
only countably many connected components.
Without loss of generality, we may assume that the coverings  we consider are  locally finite. We denote
by $\scE^{(p,q)}_{X}$ and $\scO^{(p)}_{X}$, \r, the sheaves of  $C^\infty$ $(p,q)$-forms and of \h\ $p$-forms on $X$. We denote $\scO^{(0)}_{X}$ by $\scO_{X}$. We also omit the suffix $X$ on the sheaf notation  if there is no fear of confusion.
By the Dolbeault-Grothendieck lemma, the complex $\scE^{(p,\bullet)}$ gives a fine resolution of $\scO^{(p)}$\,:
\[
0\lra\scO^{(p)}\lra\scE^{(p,0)}\overset{\bp}\lra\scE^{(p,1)}\overset{\bp}\lra\cdots\overset{\bp}\lra 
\scE^{(p,n)}\lra 0.
\]

The {\em Dolbeault cohomology} $H^{p,q}_{\bp}(X)$ of $X$ of type $(p,q)$ is the $q$-th cohomology
of the complex $(\scE^{(p,\bullet)}(X),\bp)$.
The ``de~Rham type theorem'' in \cite[Section 2]{Su11}
reads
(cf. Remark~\ref{remcano} below)\,:

\begin{theorem}[Canonical Dolbeault theorem]\label{canDol}
There
is a canonical \iso\,{\rm :}
\[
H^{p,q}_{\bp}(X)\simeq H^{q}(X;\scO^{(p)}).
\]
\end{theorem}

\subsection{Dolbeault cohomology of  open embeddings}\label{ssoemb}

We recall the contents of \cite[Subsection 2.3]{Su11} in our situation.
Let $X$ be a complex \mfd\ of dimension $n$ as above. For an open set $X'$   in $X$ with inclusion $i:X'\hra X$, we define a complex $\scE^{(p,\bullet)}(i)$ as follows.
We set
\[
\scE^{(p,q)}(i)=\scE^{(p,q)}(X)\oplus \scE^{(p,q-1)}(X')
\]
and define the differential 
\[
\bp:\scE^{(p,q)}(i)=\scE^{(p,q)}(X)\oplus \scE^{(p,q-1)}(X')
\lra \scE^{(p,q+1)}(i)=\scE^{(p,q+1)}(X)\oplus \scE^{(p,q)}(X')
\] 
by 
\[
\bp(\o,\t)=(\bp\o,i^{*}\o-\bp\t),
\]
where $i^{*}:\scE^{(p,q)}(X)\ra \scE^{(p,q)}(X')$ denotes the pull-back of differential forms by $i$, the restriction
to $X'$ in this case. Obviously we have $\bp\circ \bp=0$.

\begin{definition} The {\em Dolbeault 
cohomology $H^{p,q}_{\bp}(i)$ of $i:X'\ra X$}
 is the
 cohomology of $(\scE^{(p,\bullet)}(i), \bp)$.
\end{definition}

Denoting by $\scE^{(p,\bullet)}[-1]$ the complex with $\scE^{(p,q)}[-1]=\scE^{(p,q-1)}$ and 
the differential $-\bp$,
we  define morphisms  $\a^{*}:\scE^{(p,\bullet)}(i)\ra \scE^{(p,\bullet)}(X)$ and $\b^{*}:\scE^{(p,\bullet)}[-1](X')
\ra \scE^{(p,\bullet)}(i)$
by
\[
\begin{aligned}
&{\a^{*}}:\scE^{(p,q)}(i)=\scE^{(p,q)}(X)\oplus \scE^{(p,q-1)}(X')\lra \scE^{(p,q)}(X),\qquad (\o,\t)\mapsto \o,\qquad\text{and}\\ 
&\b^{*}:\scE^{(p,q)}[-1](X')=\scE^{(p,q-1)}(X')\lra \scE^{(p,q)}(i)=\scE^{(p,q)}(X)\oplus \scE^{(p,q-1)}(X'),\qquad \t\mapsto (0,\t).
\end{aligned}
\]
Then we have the exact sequence of complexes
\begin{equationth}\label{sexactcomemb}
0\lra \scE^{(p,\bullet)}[-1](X')\overset{\b^{*}}\lra \scE^{(p,\bullet)}(i)\overset{\a^{*}}\lra \scE^{(p,\bullet)}(X)\lra 0,
\end{equationth}
which gives rise to the exact sequence
\begin{equationth}\label{exactcomemb}
\cdots\lra H^{p,q-1}_{\bp}(X')\overset{\b^{*}}\lra H^{p,q}_{\bar\vt}(i)\overset{\a^{*}}\lra 
H^{p,q}_{\bp}(X)\overset{i^{*}}\lra H^{p,q}_{\bp}(X')\lra\cdots.
\end{equationth}

The ``relative de Rham type theorem'' in \cite[Subsection 2.3]{Su11} reads in our case\,:

\begin{theorem}[Relative Dolbeault theorem]\label{thdrel}
There is a canonical \iso\,{\rm :}
\[
H^{p,q}_{\bar\vt}(i)\simeq H^{q}(X,X';\scO^{(p)}).
\]
\end{theorem}

\begin{remark}\label{remBT} {\bf 1.}  The above cohomology $H^{p,q}_{\bar\vt}(i)$ has already appeared in a number of literatures, e.g., \cite{HT} and \cite{Id}. For the de~Rham complex it is introduced in  \cite{BT} in a little more general setting (cf. Remark~\ref{remcohmap}.\,1 below). 
\smallskip

\noindent
{\bf 2.} The complex $\scE^{(p,\bullet)}(i)$ is nothing but the ``co-mapping cone'' $M^{*}(i^{*})$ of the
morphism $i^{*}:\scE^{(p,\bullet)}(X)\ra\scE^{(p,\bullet)}(X')$ (cf. \cite[Section 5]{Su11}). It is also identical with the complex $\scE^{(p,\bullet)}(\V^{\star},\V')$ considered in Section \ref{secrDol} below and the cohomology $H^{p,q}_{\bp}(i)$ is identical with $H^{p,q}_{\bar\vt}(X,X')$,  the {\em relative Dolbeault cohomology} 
of $(X,X')$ (cf. \eqref{rel=cmc}).
\smallskip

\noindent
{\bf 3.} If we follow the notation of \cite{Su11}, $\scE^{(p,\bullet)}(i)$ should be denoted something like
$\scE(i)^{(p,\bullet)}$. The same remark applies to the notation such as $\scE^{(p,\bullet)}(\W,\W')$ in the subsequent sections.
%
\end{remark}

\paragraph{Dolbeault cohomology of  \h\ maps\,:}

Let $f:Y\ra X$ be a \h\ map of complex \mfd s. We may directly generalize the above construction to this situation, replacing $X'$ and $i$ by $Y$ and $f$.  Thus we set
\[
\scE^{(p,q)}(f)=\scE^{(p,q)}_{X}(X)\oplus \scE^{(p,q-1)}_{Y}(Y)
\]
and define $\bp:\scE^{(p,q)}(f)\ra \scE^{(p,q+1)}(f)$ by
\[
\bp(\o,\t)=(\bp\o,f^{*}\o-\bp\t).
\]
Then $(\scE^{(p,\bullet)}(f),\bp)$ is a complex. 

\begin{definition} The Dolbeault cohomology $H^{p,q}_{\bp}(f)$ of $f$ of type $(p,q)$ is defined as the 
$q$-th cohomology of $(\scE^{(p,\bullet)}(f),\bp)$.
\end{definition}

We denote by $\scE^{(p,\bullet)}_{Y}(Y)[-1]$ the complex \st\ $\scE^{(p,\bullet)}_{Y}(Y)[-1]^{q}=\scE^{(p,q-1)}(Y)$ with
the differential given by $-\bp$. Then we have the exact sequence of complexes
\[
0\lra \scE^{(p,\bullet)}_{Y}(Y)[-1]^{q}\overset{\b^{*}}\lra \scE^{(p,\bullet)}(f)\overset{\a^{*}}\lra \scE^{(p,\bullet)}_{X}(X)\lra 0,
\]
where $\a^{*}(\o,\t)=\o$ and $\b^{*}(\t)=(0,\t)$. Then we have the exact sequence
\[
\cdots \lra H^{p,q-1}_{\bp}(Y)\overset{\b^{*}}\lra H^{p,q}_{\bp}(f)\overset{\a^{*}}\lra H^{p,q}_{\bp}(X)
\overset{f^{*}}\lra H^{p,q}_{\bp}(Y)\lra\cdots.
\]


In the case $Y=X'$ is an open set in $X$ and $f=i$ is the inclusion, the above cohomology is nothing but $H^{p,q}_{\bp}(i)$  defined before.

\begin{remark}\label{remcohmap} {\bf 1.} Similar construction is done in \cite{BT} for the de~Rham case.
\smallskip

\noindent
{\bf 2.} The complex $\scE^{(p,\bullet)}(f)$ is nothing but the co-mapping cone $M^{*}(f^{*})$ of the
morphism $f^{*}:\scE^{(p,\bullet)}_{X}(X)\ra\scE^{(p,\bullet)}_{Y}(Y)$ (cf. \cite[Section 5]{Su11}).
\end{remark}

There is the notion of the cohomology of a sheaf morphism (cf. \cite[Section 6]{Su11} and references therein).
In our case it is defined as follows. We first consider the space $Z(f)=X\amalg Y$ (disjoint union). For an open set $U$ in $X$, we set  $\tilde U=U\amalg f^{-1}U$
and endow $Z(f)$
with the topology whose basis of open sets consists of 
$\{\, \tilde U\mid U\subset X\ \text{open sets}\,\}$ and $\{\,V\mid V\subset Y\ \text{open sets}\,\}$.
Then we have the closed embedding $X\hra Z(f)$ and the open embedding $Y\hra Z(f)$.
Recall in general that, for a sheaf $\scT$ on $Y$,  the direct image $f_{*}\scT$ is the sheaf on $X$ defined by the presheaf
$U\mapsto \scT(f^{-1}U)$. In our situation, there is the sheaf morphism $f^{*}:\scO^{(p)}_{X}\ra f_{*}\scO^{(p)}_{Y}$
given by the pull-back of differential forms. Let $\scZ^{*}(f^{*})=\scZ^{*}(\scO^{(p)}_{Y}\overset{f^{*}}\leftarrow \scO^{(p)}_{X})$ be the sheaf on $Z(f)$ defined by 
the presheaf $\tilde U\mapsto \scO^{(p)}_{X}(U)$
and $V\mapsto \scO^{(p)}_{Y}(V)$. The presheaf is a sheaf, i.e., 
 $\scZ^{*}(f^{*})(\tilde U)=\scO^{(p)}_{X}(U)$
 and
$\scZ^{*}(f^{*})(V)=\scO^{(p)}_{Y}(V)$. The restriction $\scZ^{*}(f^{*})(\tilde U)=\scO^{(p)}_{X}(U)\ra \scZ^{*}(f^{*})(f^{-1}U)=
\scO^{(p)}_{Y}(f^{-1}U)$
is given by $f^{*}$. Then the cohomology $H^{q}(f;f^{*})=H^{q}(Y\overset f \ra X;\scO^{(p)}_{Y}\overset{f^{*}}\leftarrow \scO^{(p)}_{X})$ of the morphism $f^{*}:\scO^{(p)}_{X}\ra f_{*}\scO^{(p)}_{Y}$
is defined by (cf. \cite[Section 6]{Su11})
\[
H^{q}(Y\overset f \ra X;\scO^{(p)}_{Y}\overset{f^{*}}\leftarrow \scO^{(p)}_{X})=H^{q}(Z(f),Z(f)\ssm X; \scZ^{*}(\scO^{(p)}_{X}\overset{f^{*}}\ra \scO^{(p)}_{Y})).
\]
There is an exact sequence\,:
\[
\cdots\lra H^{q-1}(Y;\scO^{(p)}_{Y})\lra H^{q}(f;f^{*})\lra H^{q}(X;\scO^{(p)}_{X})\lra H^{q}(Y;\scO^{(p)}_{Y})\lra\cdots.
\]

In the case $f:Y\hra X$ is an open embedding, we may identify $f^{*}$ with the pull-back $f^{-1}$ of sections
and we write $H^{q}(Y\overset f \ra X;\scO^{(p)}_{Y}\overset{f^{*}}\leftarrow \scO^{(p)}_{X})$ as $H^{q}(f;\scO^{(p)}_{X})$. Thus we have (cf. loc. cit)\,:
\begin{proposition}\label{propoe3} In the case $f:Y\hra X$ is an open embedding, 
there is a canonical \iso 
\[
H^{q}(f;\scO^{(p)}_{X})\simeq H^{q}(X,Y;\scO^{(p)}_{X}).
\]
\end{proposition} 

In general, since we have the commutative the diagram
\[
\SelectTips{cm}{}
\xymatrix@C=.7cm
@R=.6cm
{ 0\ar[r]& \scO^{(p)}_{X}\ar[r]\ar[d]^-{f^{*}}&\scE^{(p,\bullet)}_{X}\ar[d]^-{f^{*}}\\
0\ar[r] & f_{*}\scO^{(p)}_{Y} \ar[r] & f_{*}\scE^{(p,\bullet)}_{Y},}
\]
 we have\,:
\begin{theorem}[Generalized relative Dolbeault theorem]\label{thgrD} For a \h\ map $f:Y\ra X$ of complex
\mfd s, there is a canonical \iso\,{\rm :}
\[
H^{p,q}_{\bp}(f)\simeq H^{q}(Y\overset f \ra X;\scO^{(p)}_{Y}\overset{f^{*}}\leftarrow \scO^{(p)}_{X}).
\]
\end{theorem}

In the case $f:Y\hra X$ is an open embedding, the above reduces to Theorem \ref{thdrel}.
\lsection{\v{C}ech-Dolbeault cohomology}\label{secCD}

We recall the contents of \cite[Section 3]{Su11}  specializing them to our setting.

\subsection{\v{C}ech cohomology}


Let $X$ be a topological space $\scS$ a sheaf on $X$ and $\W=\{W_\a\}_{\a\in I}$ 
an open covering of $X$. 
We set $W_{\a_0\dots\a_q}=W_{\a_0}\cap\dots\cap W_{\a_q}$ and consider  the direct product 
\[
C^{q}(\W; \scS)=\prod_{(\a_0,\dots,\a_{q})\in I^{q+1}}\scS(W_{\a_0\dots\a_{q}}).
\]
The $q$-th \v{C}ech cohomology $H^{q}(\W;\scS)$ of $\scS$ on $\W$  is the $q$-th cohomology
of the complex $(C^{\bullet}(\W; \scS),\check\delta)$ with $\check\delta:C^{q}(\W; \scS)\ra C^{q+1}(\W; \scS)$ 
defined by
\[
(\check\delta\sigma)_{\a_0\dots \a_{q+1}}
=\sum_{\nu=0}^{q+1}(-1)^\nu\sigma_{\a_0\dots\widehat{\a_{\nu}}\dots\a_{q+1}}.
\]

Let  $X'$ be an open set in $X$. Let $\W=\{W_{\a}\}_{\a\in I}$ be a covering 
of $X$ \st\ $\W'=\{W_{\a}\}_{\a\in I'}$ is a covering  of $X'$ for some $I'\subset I$.
We set
\[
C^{q}(\W, \W';\scS)
=\{\,\sigma\in C^{q}(\W; \scS)\mid \sigma_{\a_0\dots \a_{q}}=0\ \ \text{if}\ \ 
\a_0,\dots ,\a_{q}\in I'\,\}
\]
The  operator $\check\delta$ restricts to 
$C^{q}(\W,  \W';\scS)\to C^{q+1}(\W,  \W';\scS)$. 
The $q$-th \v{C}ech cohomology $H^{q}(\W,\W';\scS)$ of $\scS$ on $(\W,\W')$  is the $q$-th cohomology
of 
$(C^{\bullet}(\W,\W'; \scS),\check\delta)$. 

We have the following\,:
\begin{theorem}[Relative Leray theorem]\label{thleray} If $H^{q_{2}}(W_{\a_0\dots\a_{q_{1}}},\scS)=0$
 for $q_{1}\ge 0$ and $q_{2}\ge 1$, there is a canonical \iso
\[
H^{q}(\W,\W';\scS)\simeq H^{q}(X,X';\scS).
\]
\end{theorem}

\subsection{\v{C}ech-Dolbeault cohomology}\label{ssecCD}
We review the contents of 
 \cite[Subsection 3.2]{Su11} in our case.

Let $X$,  $\scE^{(p,q)}_{X}$ and $\scO^{(p)}_{X}$ be as in Subsection \ref{ssDol}.
Also let 
$X'$ be an open set in $X$ and let $\W$ and $\W'$ be coverings of $X$ and $X'$ as before. 
Then
we have a double complex $(C^\bullet(\W,\W';\scE^{(p,\bullet)}),\check\delta,(-1)^{\bullet}\bp)$. We consider the associated single complex 
$(\scE^{(p,\bullet)}(\W,\W'), \bar\vt)$.
Thus
\[
\scE^{(p,q)}(\W,\W')=\bigoplus_{q_{1}+q_{2}=q}C^{q_{1}}(\W,\W';\scE^{(p,q_{2})}),\qquad \bar\vt=\check\delta+(-1)^{q_{1}}\bp.
\]

\begin{definition}\label{defCD} The {\em \v{C}ech-Dolbeault cohomology} $H^{p,q}_{\bar\vt}(\W,\W')$ of type $(p,q)$ on $(\W,\W')$ is the $q$-th cohomology
of the complex $(\scE^{(p,\bullet)}(\W,\W'), \bar\vt)$.
\end{definition}

In the case $X'=\emptyset$, we take $\emptyset$ as  $I'$ and
denote $\scE^{(p,\bullet)}(\W,\W')$ and $H^{p,q}_{\bar\vt}(\W,\W')$  by $\scE^{(p,\bullet)}(\W)$ and $H^{p,q}_{\bar\vt}(\W)$.

We recall the description of the differential $\bar\vt$ as given in \cite{Su11}.
Note that a  cochain $\xi$ in $\scE^{(p,q)}(\W,\W')$
 may be expressed as 
$\xi=(\xi^{q_{1}})_{0\le q_{1}\le q}$ with $\xi^{q_{1}}$ in $C^{q_{1}}(\W,\W';\scE^{(p,q-q_{1})})$. In the sequel $\xi^{q_{1}}_{\a_0\dots \a_{q_{1}}}$ is also written as $\xi_{\a_0\dots \a_{q_{1}}}$.
Then   $\bar\vt:\scE^{(p,q)}(\W,\W')\ra \scE^{(p,q+1)}(\W,\W')$ is given by
\begin{equationth}\label{cdrdiff}
(\bar\vt\xi)^{q_{1}}=\check\delta\xi^{q_{1}-1}+(-1)^{q_{1}}\bp\xi^{q_{1}},\qquad 0\le q_{1}\le q+1,
\end{equationth}
where we set $\xi^{-1}=0$ and $\xi^{q+1}=0$.
In particular, for $q_{1}=0,1$,
\begin{equationth}\label{small}
(\bar\vt\xi)_{\a_0}=\bp\xi_{\a_0},\qquad (\bar\vt\xi)_{\a_0\a_1}
=\xi_{\a_1}-\xi_{\a_0}-\bp\xi_{\a_0\a_1}.
\end{equationth}
Thus the  condition for $\xi$ being a cocycle is given by
\[
\begin{cases}
  \bp\xi^{0}=0,\\
  \check\delta\xi^{q_{1}-1}+(-1)^{q_{1}}\bp\xi^{q_{1}}=0,\qquad 1\le q_{1}\le q,\\
  \check\delta\xi^{q}=0.
\end{cases}
\]

We have $H^{p,0}_{\bar\vt}(\W,\W')=\scO^{(p)}(X,X')$.

For a triple $(\W,\W',\W'')$, we have the exact sequence
\begin{equationth}\label{sexactcd}
0\lra \scE^{(p,\bullet)}(\W,\W')\lra \scE^{(p,\bullet)}(\W,\W'')\lra \scE^{(p,\bullet)}(\W',\W'')\lra 0
\end{equationth}
yielding an exact sequence
\begin{equationth}\label{lexactcd}
\cdots\lra H^{p,q-1}_{\bar\vt}(\W',\W'')\overset{\delta}\lra H^{p,q}_{\bar\vt}(\W,\W')
\lra H^{p,q}_{\bar\vt}(\W,\W'')
\lra H^{p,q}_{\bar\vt}(\W',\W'')\lra\cdots.
\end{equationth}





\begin{remark}\label{remalt}  We may use only ``alternating cochains'' in the above construction and the resulting
cohomology is canonically isomorphic with the one defined above.
\end{remark}

\paragraph{Some special cases\,:} {\bf I.} In the case $\W=\{X\}$, we have $(\scE^{(p,\bullet)}(\W),\bar\vt)=(\scE^{(p,\bullet)}(X),\bp)$ and
$H^{p,q}_{\bar\vt}(\W)=H^{p,q}_{\bp}(X)$.
\vv

\noindent
{\bf II.} In the case $\W$ consists of two open sets $W_{0}$ and $W_{1}$, we may write (cf. Remark \ref{remalt})
\[
\scE^{(p,q)}(\W)=C^{0}(\W,\scE^{(p,q)})\oplus C^{1}(\W,\scE^{(p,q-1)})=\scE^{(p,q)}(W_{0})\oplus  \scE^{(p,q)}(W_{1})\oplus \scE^{(p,q-1)}(W_{01}).
\]
Thus a cochain $\xi\in \scE^{(p,q)}(\W)$ is expressed as a triple $\xi=(\xi_{0},\xi_{1},\xi_{01})$
and the differential 
\[
\bar\vt:\scE^{(p,q)}(\W)\ra \scE^{p,q+1}(\W)\quad\text{is given by}\quad 
\bar\vt(\xi_{0},\xi_{1},\xi_{01})=(\bp\xi_{0},\bp\xi_{1},\xi_{1}-\xi_{0}-\bp\xi_{01}).
\]

If we set $Z^{p,q}(\W)=\op{Ker}\bar\vt^{p,q}$ and $B^{p,q}(\W)=\op{Im}\bar\vt^{p,q-1}$, then  by definition,
$H^{p,q}_{\bar\vt}(\W)=Z^{p,q}(\W)/B^{p,q}(\W)$. We may somewhat simplify the coboundary
group $B^{p,q}(\W)$ (cf. \cite{Su11})\,:

\begin{proposition} We have
\[
B^{p,q}(\W)=\{\,\xi\in \scE^{(p,q)}(\W)\mid \xi=(\bp\eta_{0},\bp\eta_{1},\eta_{1}-\eta_{0}),\
\text{for some}\ \eta_{i}\in \scE^{p,q-1}(W_{i}), i=0, 1\,\}.
\]
\end{proposition}

In the relative case, if we set $\W'=\{W_{0}\}$, then
\[
\scE^{(p,q)}(\W,\W')=\{\,\xi\in \scE^{(p,q)}(\W)\mid \xi_{0}=0\,\}= \scE^{(p,q)}(W_{1})\oplus \scE^{p,q-1}(W_{01}).
\]
Thus a cochain $\xi \in \scE^{(p,q)}(\W,\W')$ is expressed as a pair $\xi=(\xi_{1},\xi_{01})$ and the differential
\[
\bar\vt:\scE^{(p,q)}(\W,\W')\ra \scE^{p,q+1}(\W,\W')\quad\text{is given by}\quad \bar\vt(\xi_{1},\xi_{01})=(\bp\xi_{1},\xi_{1}-\bp\xi_{01}). 
\]
The $q$-th cohomology of $(\scE^{p,\bullet}(\W,\W'),\bar\vt)$ is $H^{p,q}_{\bar\vt}(\W,\W')$.

If we set $\W''=\emptyset$, then $H^{p,q-1}_{\bar\vt}(\W',\W'')=H^{p,q-1}_{\bar\vt}(\W')=H^{p,q-1}_{\bp}(W_{0})$ and the connecting morphism $\delta$ in (\ref{lexactcd}) assigns to the class of a $\bp$-closed form
$\xi_{0}$ on $W_{0}$ the class of $(0,-\xi_{0})$ (restricted to $W_{1}$) in $H^{p,q}_{\bar\vt}(\W,\W')$.

We discuss this case more in detail in the subsequent section.
\vv

\noindent
{\bf III.} Suppose $\W$ consists of three open sets $W_{0}$, $W_{1}$ and $W_{2}$ and set
$\W'=\{W_{0},W_{1}\}$ and $\W''=\{W_{0}\}$. Then
\[
\begin{aligned}
\scE^{(p,q)}(\W)&=\textstyle\bigoplus_{i=0}^{2}\scE^{(p,q)}(W_{i})\oplus\textstyle\bigoplus_{0\le i<j\le 2}
\scE^{(p,q-1)}(W_{ij})\oplus \scE^{(p,q-2)}(W_{012}),\\
\scE^{(p,q)}(\W,\W'')&=\textstyle\bigoplus_{i=1}^{2}\scE^{(p,q)}(W_{i})\oplus\textstyle\bigoplus_{0\le i<j\le 2}
\scE^{(p,q-1)}(W_{ij})\oplus \scE^{(p,q-2)}(W_{012}),\\
\scE^{(p,q)}(\W,\W')&=\scE^{(p,q)}(W_{2})\oplus
\scE^{(p,q-1)}(W_{02})\oplus
\scE^{(p,q-1)}(W_{12})\oplus \scE^{(p,q-2)}(W_{012}),\\
\scE^{(p,q)}(\W',\W'')&=\scE^{(p,q)}(W_{1})\oplus \scE^{(p,q-1)}(W_{01}).
\end{aligned}
\]
The connecting morphism $\delta$ in (\ref{lexactcd}) assigns to the class of 
$(\t_{1},\t_{01})$ in $H^{p.q-1}_{\bar\vt}(\W',\W'')$ the class of $(0,0,-\t_{1},\t_{01})$ (restricted to $W_{2}$) in $H^{p,q}_{\bar\vt}(\W,\W')$.

\paragraph{Canonical \iso s\,:}

We say that a covering $\W=\{W_\a\}$ of $X$ is  {\em Stein}, if   every non-empty finite intersection $W_{\a_0\dots\a_{q_{1}}}$ is a Stein \mfd.
In fact, for this it is sufficient
if each $W_{\a}$ is Stein (cf. \cite{GrR}, \cite{N}).
Note that every complex \mfd\ $X$
admits a Stein covering and that the Stein coverings are cofinal in the set of coverings of $X$. We quote\,:
\begin{theorem}[Oka-Cartan] For any coherent sheaf $\scS$ on a Stein \mfd\ $W$,
\[
H^{q}(W;\scS)=0\qquad\text{for}\ \ q\ge 1.
\]
\end{theorem}

By the above and 
Theorem \ref{canDol}, we see that a Stein covering is good for $\scE^{(p,\bullet)}$ in the sense
of \cite[Section 3]{Su11}.
Thus in our case, we have\,:

\begin{theorem}\label{thsummary} We have the following canonical \iso s\,{\rm :}
\smallskip

\noindent
{\bf 1.}  For any covering $\W$,
\[
H^{p,q}_{\bp}(X)\overset\sim\lra H^{p,q}_{\bar\vt}(\W).
\]

\noindent
{\bf 2.} For a Stein covering $\W$,
\[
H^{p,q}_{\bar\vt}(\W,\W')\overset\sim\longleftarrow H^{q}(\W,\W';\scO^{(p)})\simeq H^{q}(X,X';\scO^{(p)}).
\]
\end{theorem}
\begin{remark}\label{remcan} {\bf 1.} From 1 above we see that  $H^{p,q}_{\bar\vt}(\W)$ does not depend on the covering $\W$. The \iso\ there is induced from the inclusion of complexes\,:
\[
\scE^{(p,\bullet)}(X)\hra C^0(\W;\scE^{(p,\bullet)})\subset \scE^{(p,\bullet)}(\W).
\]

In particular,  if $W_{\a}=X$ for some $\a\in I$, it can be shown that 
the morphism $\scE^{(p,\bullet)}(\W)\ra\scE^{(p,\bullet)}(X)$ 
given by $\xi\mapsto\xi_{\a}$ induces  the inverse of the above \iso\ (cf. \cite{Su11}).
See also Propositions \ref{propinvtwo} and \ref{propinvthree} below.
\smallskip

\noindent
{\bf 2.} The first \iso\ in 2 above is induced from the inclusion of complexes\,:
\[
C^{\bullet}(\W,\W';\scO^{(p)})\hra C^{\bullet}(\W,\W';\scE^{(p,0)})\subset \scE^{(p,\bullet)}(\W,\W').
\]
 The second \iso\ follows from Theorem \ref{thleray}.
\end{remark}

From Theorem \ref{thsummary} we have\,:

\begin{corollary}\label{natisos}
If $\W$ is Stein,
there is a 
canonical \iso\,{\rm : }
\[
H^{p,q}_{\bp}(X)\simeq
H^{q}(\W;\scO^{(p)}).
\]
\end{corollary}

In the above, we think of a  Dolbeault cocycle $\o\in \scE^{(p,q)}(X)$  and a \v{C}ech cocycle $c\in C^{q}(\W;\scO^{(p)})$  as being \v{C}ech-Dolbeault cocycles, i.e., cocycles in $\scE^{(p,q)}(\W)$.
The classes $[\o]\in H^{q}_{\bp}(X)$ and $[c]\in H^{q}(\W;\scO^{(p)})$ correspond in the above \iso, \iff\ $\o$ and $c$ define the same class in 
$H^{q}_{\bar\vt}(\W)$, i.e., there exists a $(q-1)$-cochain $\chi\in \scE^{(p,q-1)}(\W)$
 \st
\[
\o-c=\bar\vt\chi.
\]
The above relation is rephrased as, for $\chi^{q_{1}}$ in 
$C^{q_{1}}(\W;\scE^{(p,q-q_{1}-1)})$, $0\le q_{1}\le q-1$,
\begin{equationth}\label{dRCcorr}
\begin{cases}
\o=\bp\chi^{0},\\
0= \check\delta\chi^{q_{1}-1}+(-1)^{q_{1}}\bp\chi^{q_{1}},\qquad 1\le q_{1}\le q-1\\
 -c=\check\delta\chi^{q-1}.
\end{cases}
\end{equationth}

Note that the composition of the \iso\ of Corollary \ref{natisos} and the  second \iso\  of
Theorem~\ref{thsummary}.\,2 for $X'=\emptyset$ is equal to the \iso\ in Theorem \ref{canDol}.

\begin{remark}\label{remcano} 
%
 It is possible to establish an \iso\
as in Corollary \ref{natisos} without introducing the \v{C}ech-Dolbeault cohomology,  
using the so-called Weil lemma instead. 
However this correspondence is different from the one in Corollary \ref{natisos}, the difference being the sign of $(-1)^{\frac{q(q+1)}2}$, see \cite[Section 3]{Su11} for details.

The seemingly standard proof  in the textbooks, e.g., 
\cite{GH}, \cite{Hir}, of the \iso\ as in Theorem \ref{canDol} or Corollary \ref{natisos}
gives a correspondence same as the one given by the Weil lemma. Thus there is a sign difference as above.
For example, in Theorem \ref{BMCauchy} below, the sign $(-1)^{\frac{n(n-1)} 2}$ does not appear this way
(cf. \cite{GH}, \cite{H2}).
\end{remark}


We finish this section by discussing the \iso\ of Theorem \ref{thsummary}.\,1 in some special cases.
Recall that it is induced by the inclusion $\scE^{(p,q)}(X)\hra C^0(\W;\scE^{(p,q)})\subset \scE^{(p,q)}(\W)$.

In the case $\W=\{W_{0},W_{1}\}$ (cf. the case II above),
\[
\scE^{(p,q)}(\W)=\scE^{(p,q)}(W_{0})\oplus \scE^{(p,q)}(W_{1})\oplus \scE^{(p,q-1)}(W_{01})
\]
and the inclusion is given by $\o\mapsto (\o|_{W_{0}},\o|_{W_{1}},0)$. 

\begin{proposition}\label{propinvtwo} In the case $\W=\{W_{0},W_{1}\}$, the inverse of the above \iso\ is given by
assigning to the class of $\xi$ the class of 
$\o=\rho_{0}\xi_{0}+\rho_{1}\xi_{1}-\bp\rho_{0}\wedge\xi_{01}$, where $(\rho_{0},\rho_{1})$ is 
a partiton of unity subordinate to $\W$.
\end{proposition}

\begin{proof} 
Recall that  $\o$ is given by $\xi_{1}-\bp(\rho_{0}\xi_{01})$ on $W_{1}$ (cf. \cite[Section 3]{Su11}).
Using the 
the cocycle condition $\xi_{1}-\xi_{0}-\bp\xi_{01}=0$, it
can be written as  $\rho_{0}\xi_{0}+\rho_{1}\xi_{1}-\bp\rho_{0}\wedge\xi_{01}$, which is a global expression of $\o$.
\end{proof}

Likewise we may prove (cf. the case III above)\,:

\begin{proposition}\label{propinvthree} In the case $\W=\{W_{0},W_{1},W_{2}\}$, the inverse of the above \iso\ is given by
assigning to the class of $\xi$ the class of 
\[
\o=\sum_{i=0}^{2}\rho_{i}\xi_{i}+\sum_{0\le i<j\le 2}(\rho_{i}\bp\rho_{j}-\rho_{j}\bp\rho_{i})\wedge\xi_{ij}
+(\bp\rho_{0}\wedge\bp\rho_{1}+\bp\rho_{1}\wedge\bp\rho_{2})\wedge\xi_{012},
\]
where $\{\rho_{0},\rho_{1},\rho_{2}\}$ is a partition of unity subordinate to $\W$.
\end{proposition}
\lsection{Relative Dolbeault cohomology}\label{secrDol}

We specialize the contents of  \cite[Section 4]{Su11} to our setting.

Let $X$ be a complex \mfd\  and $X'$  an open set in $X$. 
Letting $V_{0}=X'$ and $V_{1}$ a \nbd\ of the closed set $S=X\ssm X'$, consider the coverings $\V=\{V_{0},V_{1}\}$ and 
$\V'=\{V_{0}\}$ of $X$ and $X'$  (cf. the case II in Section \ref{secCD}). 
We have the cohomology $H^{p,q}_{\bar\vt}(\V,\V')$ as 
 the cohomology of the complex $(\scE^{(p,\bullet)}(\V,\V'),\bar\vt)$, where
\[
\scE^{(p,q)}(\V,\V')=\scE^{(p,q)}(V_{1})\oplus \scE^{(p,q-1)}(V_{01}),\qquad V_{01}=V_{0}\cap V_{1},
\]
and $\bar\vt:\scE^{(p,q)}(\V,\V')\ra \scE^{(p,q+1)}(\V,\V')$ is given by 
$\bar\vt(\xi_{1},\xi_{01})=(\bp\xi_{1},\xi_{1}-\bp\xi_{01})$.
Noting that $\scE^{(p,q)}(\{V_{0}\})=\scE^{(p,q)}(X')$,
we have the exact sequence
\begin{equationth}\label{sexactreld}
0\lra \scE^{(p,\bullet)}(\V,\V')\overset{j^{*}}\lra \scE^{(p,\bullet)}(\V)\overset{i^{*}}\lra \scE^{(p,\bullet)}(X')\lra 0,
\end{equationth}
where $j^{*}(\xi_{1},\xi_{01})=(0,\xi_{1},\xi_{01})$ and $i^{*}(\xi_{0},\xi_{1},\xi_{01})=\xi_{0}$.  This gives rise to the exact sequence
(cf. \eqref{lexactcd})
\begin{equationth}\label{lexactrelD}
\cdots\lra H^{p,q-1}_{\bp}(X')\overset{\delta}\lra H^{p,q}_{\bar\vt}(\V,\V')\overset{j^{*}}\lra H^{p,q}_{\bar\vt}(\V)\overset{i^{*}}\lra H^{p,q}_{\bp}(X')\lra\cdots,
\end{equationth}
where $\delta$ assigns to the class of $\t$ the class of $(0,-\t)$.

Now we consider the special case where $V_{1}=X$.
Thus, letting $V_{0}=X'$ and $V^{\star}_{1}=X$,
we  consider
the coverings $\V^{\star}=\{V_{0},V^{\star}_{1}\}$ and $\V'=\{V_{0}\}$ of $X$ and $X'$.

\begin{definition}\label{defrelcose} We denote $H^{p,q}_{\bar\vt}(\V^{\star},\V')$  by $H^{p,q}_{\bar\vt}(X,X')$ and call it the 
 {\em relative Dolbeault cohomology} of  $(X,X')$.
\end{definition}

In the case $X'=\emptyset$, it coincides with $H^{p,q}_{\bp}(X)$.
If we denote by $i:X'\hra X$ the inclusion, by construction we see that (cf. Subsection \ref{ssoemb})\,:
\begin{equationth}\label{rel=cmc}
\scE^{(p,\bullet)}(\V^{\star},\V')=\scE^{(p,\bullet)}(i)\quad\text{and}\quad H^{p,q}_{\bar\vt}(X,X')=H^{p,q}_{\bp}(i).
\end{equationth}

By Theorem \ref{thsummary}.\,1,
there is a canonical \iso\ 
$H^{p,q}_{\bp}(X)\overset\sim\ra H^{p,q}_{\bar\vt}(\V^{\star})$, which assigns to the class of $s$ the
class of $(s|_{X'},s,0)$ . Its inverse assigns to the class of $(\xi_{0},\xi_{1},\xi_{01})$ the class of $\xi_{1}$
(cf. Remark \ref{remcan}.\,1).
Thus from \eqref{lexactrelD} we have the  exact sequence
\begin{equationth}\label{lexactrelD2}
\cdots\lra H^{p,q-1}_{\bp}(X')\overset{\delta}\lra H^{p,q}_{\bar\vt}(X,X')\overset{j^{*}}\lra H^{p,q}_{\bp}(X)\overset{i^{*}}\lra H^{p,q}_{\bp}(X')\lra\cdots,
\end{equationth}
where $j^{*}$ assigns to the class of
$(\xi_{1},\xi_{01})$  the class of $\xi_{1}$ and $i^{*}$ assigns to the class of
$s$  the class of $s|_{X'}$. It coincides with the sequence \eqref{exactcomemb}, except $\delta=-\b^{*}$.

We have the following propositions (cf. \cite{Su11})\,:
\begin{proposition}\label{proptriplerd} For a triple $(X,X',X'')$, there is an exact sequence
\[
\cdots\lra H^{p,q-1}_{\bar\vt}(X',X'')\overset{\delta}\lra H^{p,q}_{\bar\vt}(X,X')\overset{j^{*}}\lra H^{p,q}_{\bar\vt}(X,X'')\overset{i^{*}}\lra H^{p,q}_{\bar\vt}(X',X'')\lra\cdots.
\]
\end{proposition}




Let $\V=\{V_{0},V_{1}\}$ be as in the beginning of this section, with $V_{1}$ an arbitrary
open set containing $X\ssm X'$. By Theorem~\ref{thsummary}.\,1, there is a canonical \iso\ $H^{p,q}_{D_{\scK}}(\V)\simeq H^{p,q}_{\bp}(X)$ and in
\eqref{lexactrelD}, 
$j^{*}$ assigns to the class of
$(\xi_{1},\xi_{01})$ the class of  $(0,\xi_{1},\xi_{01})$ or the class of $\rho_{1}\xi_{1}-\bp\rho_{0}\wedge\xi_{01}$ 
(or the class of $\xi_{1}$ if $V_{1}=X$) (cf. Proposition \ref{propinvtwo}, also Remark \ref{remcan}.\,1).

\begin{proposition}\label{propuni} 
The restriction $\scE^{(p,\bullet)}(\V^{\star},\V')\ra \scE^{(p,\bullet)}(\V,\V')$ induces an \iso
\[
H_{\bar\vt}^{p,q}(X,X')\overset\sim\lra H_{\bar\vt}^{p,q}(\V,\V').
\]
\end{proposition}


\begin{corollary}\label{corunique} The cohomology $H_{\bar\vt}^{p,q}(\V,\V')$ is uniquely determined modulo canonical \iso s, independently
of the choice of $V_{1}$.
\end{corollary}

\begin{remark} This freedom of choice of $V_{1}$ is one of the advantages of expressing 
$H^{p,q}_{\bp}(i)$ as $H_{\bar\vt}^{p,q}(X,X')$.
\end{remark}


\begin{proposition}[Excision]\label{excision} Let $S$ be a closed set in $X$. Then, for any open set $V$ in $X$ containing $S$, there is a canonical \iso
\[
H^{p,q}_{\bar\vt}(X,X\ssm S)\overset{\sim}{\lra} H^{p,q}_{\bar\vt}(V,V\ssm S).
\]
\end{proposition}

Now we indicate an alternative proof of Theorem \ref{thdrel} and refer to \cite{Su11} for details.
Let $\W=\{W_{\a}\}_{\a\in I}$ be a covering of $X$ and $\W'=\{W_{\a}\}_{\a\in I'}$ a covering of $X'$, $I'\subset I$.
Letting $V^{\star}_{1}=X$ as before, we define a morphism
\[
\varphi:\scE^{(p,q)}(\V^{\star},\V')\lra C^{0}(\W,\W';\scE^{(p,q)})\oplus C^{1}(\W,\W';\scE^{(p,q-1)})\subset \scE^{(p,q)}(\W,\W')
\]
by setting, for $\xi=(\xi_{1},\xi_{01})$,
\[
\varphi(\xi)_{\a}=\begin{cases} 0\quad &\a\in I'\\
                                                           \xi_{1}|_{W_{\a}}& \a\in I\ssm I',
\end{cases}\qquad\quad
\varphi(\xi)_{\a\b}=\begin{cases} \xi_{01}|_{W_{\a\b}}\quad &\a\in I', \  \b\in I\ssm I'\\
-\xi_{01}|_{W_{\a\b}}\quad &\a\in I\ssm I', \  \b\in I'\\
                                                           0& \text{otherwise}.
\end{cases}
\]

\begin{theorem}\label{3.2rel} 
The above  morphism $\varphi$ 
induces an \iso
\[
H^{p,q}_{\bar\vt}(X,X')\overset{\sim}{\lra}  H^{p,q}_{\bar\vt}(\W,\W').
\]
\end{theorem}

Using the above we have an alternative proof of the  relative Dolbeault theorem
(Theorem \ref{thdrel})\,:

\begin{theorem}\label{thdrel2} 
There is a canonical \iso\,{\rm : }
\[
H^{p,q}_{\bar\vt}(X,X')\simeq H^{q}(X,X';\scO^{(p)}).
\]
\end{theorem}



The sequence in Proposition \ref{proptriplerd} is
compatible with the corresponding sequence for the relative cohomology
and the excision of Proposition \ref{excision} is compatible with that of the relative cohomology,
both via the \iso\ of Theorem \ref{thdrel2}.  

We finish this section by presenting the following topic\,:

\paragraph{Differential\,:} 

Let $X$ be a complex \mfd\ of dimension $n$ and $X'$ an open set in $X$. We consider coverings $\W$ and $\W'$ of
$X$ and $X'$ as before.

First note that the second \iso\ of Theorem \ref{thsummary}.\,2 is compatible with the differential $d:\scO^{(p)}\ra\scO^{(p+1)}$, in fact $d=\partial$ in this case.
We define
\[
\partial:\scE^{(p,q)}(\W,\W')\lra \scE^{(p+1,q)}(\W,\W')\qquad\text{by}\ \ (\partial\xi)^{q_{1}}=(-1)^{q-q_{1}}\,\partial\xi^{q_{1}},\ \ 0\le q_{1}\le q.
\]
Straightforward computations show that it is compatible with the operator $\bar\vt$, i.e., the 
following diagram is commutative\,{\rm :}
\[
\SelectTips{cm}{}
\xymatrix
@C=.7cm
@R=.7cm
{\scE^{(p,q)}(\W,\W')\ar[r]^-{\partial} \ar[d]^-{\bar\vt}& \scE^{(p+1,q)}(\W,\W')\ar[d]^-{\bar\vt}\\
 \scE^{(p,q+1)}(\W,\W')\ar[r]^-{\partial} & \scE^{(p+1,q+1)}(\W,\W').}
\]
Thus we have
\[
\partial :H^{p,q}_{\bar\vt}(\W,\W')\lra H^{p+1,q}_{\bar\vt}(\W,\W').
\]

\begin{proposition}\label{propdcomp} If $\W$ is Stein, we have the following commutative diagram\,{\rm :}
\[
\SelectTips{cm}{}
\xymatrix
@C=.7cm
@R=.6cm
{H^{p,q}_{\bar\vt}(\W,\W')\ar[r]^-{\partial}&H^{p+1,q}_{\bar\vt}(\W,\W')\\
H^{q}(\W,\W';\scO^{(p)}) \ar[r]^-{d}  \ar[u]_-{\wr}& H^{q}(\W,\W';\scO^{(p+1)})\ar[u]_-{\wr},}
\]
where the vertical \iso s are the ones in Theorem \ref{thsummary}.\,2.
\end{proposition}

\begin{example}\label{extwoder}  In the case $\W=\{X\}$, we have $H^{p,q}_{\bar\vt}(\W)=H^{p,q}_{\bp}(X)$ (cf. the case I in Subsection \ref{ssecCD})
and 
$\partial :H^{p,q}_{\bar\vt}(\W)\ra H^{p+1,q}_{\bar\vt}(\W)$ is induced by $\t\mapsto (-1)^{q}\,\partial\t$.

In the case  $\W=\{W_{0},W_{1}\}$ (cf. Case II in Subsection \ref{ssecCD}),
$\partial :H^{p,q}_{\bar\vt}(\W)\ra H^{p+1,q}_{\bar\vt}(\W)$ is induced by
\[
(\xi_{0},\xi_{1},\xi_{01})\mapsto (-1)^{q}\,(\partial\xi_{0},\partial\xi_{1},-\partial\xi_{01}).
\]
\end{example}


From the above we have the differential
\[
\partial:H^{p,q}_{\bar\vt}(X,X')\lra H^{p+1,q}_{\bar\vt}(X,X')\quad\text{induced by}\ \ (\xi_{1},\xi_{01})\mapsto
(-1)^{q}\,(\partial\xi_{1},-\partial\xi_{01}). 
\]
From Proposition \ref{propdcomp}, we have

\begin{proposition}\label{propdcomp2} We have the following commutative diagram\,{\rm :}
\[
\SelectTips{cm}{}
\xymatrix
@C=.7cm
@R=.5cm
{H^{p,q}_{\bar\vt}(X,X')\ar[r]^-{\partial} \ar@{-}[d]^-{\wr}
&H^{p+1,q}_{\bar\vt}(X,X')\ar@{-}[d]^-{\wr}
\\
 H^{q}(X,X';\scO^{(p)})\ar[r]^-{d} & H^{q}(X,X';\scO^{(p+1)}),}
\]
where the vertical \iso s are the ones in Theorem \ref{thdrel2}.
\end{proposition}

\lsection{Relation with the case of de~Rham complex}\label{secdRD}

\subsection{Relative de~Rham cohomology}\label{ssrdr}

We refer to \cite{BT} and  \cite{Su2} for  details on  the \v{C}ech-de~Rham cohomology. For the relative de~Rham
cohomology and the Thom class in this context, see \cite{Su2}. 

In this subsection, we let $X$ denote a $C^{\infty}$ \mfd\ of dimension $m$ with a countable basis.
We assume that the coverings  we consider are  locally finite. We denote
by $\scE^{(q)}_{X}$  the sheaf of  $C^\infty$ $q$-forms on $X$. 
Recall that, by the Poincar\'e lemma, $\scE^{(\bullet)}_{X}$ gives a fine resolution of the constant
sheaf $\C_{X}$\,:
\[
0\lra\C\lra\scE^{(0)}\overset{d}\lra\scE^{(1)}\overset{d}\lra\cdots\overset{d}\lra 
\scE^{(m)}\lra 0.
\]

\paragraph{de~Rham cohomology\,:}
The  de~Rham cohomology
$H^{q}_{d}(X)$ is the  cohomology of the complex $(\scE^{(\bullet)}(X),d)$. The de~Rham theorem says that there
is an \iso
\[
H^{q}_{d}(X)\simeq H^{q}(X;\C_{X}).
\]
Note that among the \iso s, there is a canonical one (cf. \cite{Su11}).

\paragraph{\v{C}ech-de~Rham cohomology\,:}
Let $X'$ be an open set in $X$ and 
$(\W, \W')$
a pair of  coverings of $(X, X')$. The  \v{C}ech-de~Rham cohomology 
$H^{q}_{D}(\W,\W')$ on $(\W,\W')$
is 
 the cohomology of the single complex 
$(\scE^{(\bullet)}(\W,\W'),D)$ associated with
the double complex $(C^{\bullet}(\W,\W';\scE^{(\bullet})),\check\delta,(-1)^{\bullet}d)$, i.e.,
\[
\scE^{(q)}(\W,\W')=\bigoplus_{q_{1}+q_{2}=q}C^{q_{1}}(\W,\W';\scE^{(q_{2})}),\qquad  D=\check\delta+(-1)^{q_{1}}d.
\]

We say that $\W$ is {\em good} if every non-empty finite intersection $W_{\a_{0}\dots\a_{q}}$ is diffeomorphic with
$\R^{m}$. Note that every $C^{\infty}$ \mfd\ $X$
admits a good covering and that the good coverings are cofinal in the set of coverings of $X$.
By the Poincar\'e lemma, we see that a good covering is good for $\scE^{(\bullet)}$ in the sense
of \cite[Section 3]{Su11}.
Thus we have the following canonical \iso s\,:
\begin{enumerate}
\item[(1)] For any covering $\W$, $H^{q}_{d}(X)\overset\sim\ra H^{q}_{D}(\W)$.

\item[(2)] For a good covering $\W$,
\[
H^{q}_{D}(\W,\W')\overset\sim\longleftarrow H^{q}(\W,\W';\C)
%
\simeq H^{q}(X,X';\C).
\]
\end{enumerate}


\paragraph{Relative de~Rham cohomology\,:} We may also
define the relative de~Rham cohomology as in the case of relative
Dolbeault cohomology. Thus
let $S$ be a closed set in $X$. Letting $V_{0}=X\ssm S$ and $V_{1}$  a \nbd\ of $S$ in $X$, we  consider the 
coverings $\V=\{V_{0},V_{1}\}$ and $\V'=\{V_{0}\}$ of $X$  and $X\ssm S$, as before. We set
\[
\scE^{(q)}(\V,\V')=\scE^{(q)}(V_{1})\oplus\scE^{(q-1)}(V_{01})
\]
and define 
\[
D:\scE^{(q)}(\V,\V')\lra \scE^{(q+1)}(\V,\V')\quad\text{by}\quad 
D (\sigma_{1},\sigma_{01})=(d\sigma_{1},\sigma_{1}-d\sigma_{01}).
\]
\begin{definition} 
The $q$-th {\em  relative de~Rham cohomology} $H_{D}^{q}(\V,\V')$   is the 
$q$-th cohomology of the complex $(\scE^{(\bullet)}(\V,\V'),D)$. 
\end{definition}

As in the case of  Dolbeault complex, we may show that it does not
depend on the choice of $V_{1}$ and we denote it by $H_{D}^{q}(X,X\ssm S)$.
We have the relative de~Rham theorem which says that there is a canonical \iso\ (cf.  \cite{Su7}, \cite{Su11})\,:
\begin{equationth}\label{reldR}
H^{q}_{D}(X,X\ssm S)\simeq H^{q}(X,X\ssm S;\C_{X}).
\end{equationth}
\begin{remark} {\bf 1.} The sheaf cohomology
$H^{q}(X;\Z_{X})$ is canonically isomorphic with the singular 
cohomology 
$H^{q}(X;\Z)$ of $X$ with $\Z$-coefficients on finite chains and the relative sheaf cohomology $H^{q}(X,X\ssm S;\Z_{X})$ is isomorphic with the relative
singular cohomology $H^{q}(X,X\ssm S;\Z)$.
\smallskip



\noindent
{\bf 2.} In \cite{BT}, the relative de~Rham cohomology is introduced somewhat in a different way
(cf. Remark \ref{remBT}).
\end{remark}

\paragraph{Thom class\,:} Let $\pi:E\ra M$ be a $C^{\infty}$ real \vb\ of rank $l$ on a 
$C^{\infty}$ \mfd\ $M$. We identify $M$ with the image of the zero section. Suppose it is orientable as a bundle and is specified with an orientation, i.e., oriented. Then we have the Thom \iso
\[
T:H^{q-l}(M;\Z)\overset\sim\lra H^{q}(E,E\ssm M;\Z).
\]
The {\em Thom class} $\vP_{E}\in H^{l}(E,E\ssm M;\Z)$ of $E$ is the image of $[1]\in H^{0}(M;\Z)$ by $T$.



The Thom \iso\ 
with $\C$-coefficients is expressed in terms of the de~Rham and  relative
de~Rham cohomologies\,:
\[
T:H^{q-l}_{d}(M)\overset\sim\lra H^{q}_{D}(E,E\ssm M).
\]
Its inverse is given by the integration along the fibers of $\pi$ (cf. \cite[Ch.II, Theorem~5.3]{Su2}).
Let $W_{0}=E\ssm M$ and $W_{1}=E$ and consider the coverings $\W=\{W_{0},W_{1}\}$ and $\W'=\{W_{0}\}$ of  $E$ and $E\ssm M$.
Then, $H^{q}_{D}(E,E\ssm M)=H^{q}_{D}(\W,\W')$ and  
we have\,:

\begin{proposition}\label{thtrivial}
For the trivial bundle $E=\R^{l}\times M$, $\varPsi_{E}$ is represented by the cocycle
\[
(0,-\psi_{l})\qquad\text{in}\ \ \scE^{(l)}(\W,\W'),
\]
where $\psi_{l}$ is the angular form on $\R^{l}$.
\end{proposition}

Recall that $\psi_{l}$  is given by
\begin{equationth}\label{ang}
\psi_{l}=C_{l}\frac{\sum_{i=1}^{l}\varPhi_{i}(x)}{\Vert x\Vert^{l}},\qquad\varPhi_{i}(x)=(-1)^{i-1}x_{i}\,dx_{1}\wedge\cdots\wedge\widehat{dx_{i}}\wedge\cdots\wedge dx_{l}
\end{equationth}
and
\[
C_{l}=\begin{cases}\frac {(k-1)!}{2\pi^{k}}&\qquad
l=2k,\\
\frac{(2k)!}{2^{l}\pi^{k}k!}&\qquad l=2k+1.
\end{cases}
\]
The important fact is that it is a closed $(l-1)$-form and $\int_{S^{l-1}}\psi_{l}=1$ for a usually oriented $(l-1)$-sphere in $\R^{l}\ssm\{0\}$.

In the above situation,  if $M$ is orientable,  the total space $E$ is orientable. We endow them with orientations so that the orientation of
the fiber of $\pi$ followed by the orientation of $M$ gives the orientation of $E$.

Let $X$ be a $C^{\infty}$ \mfd\ of dimension $m$ and $M\subset X$ a closed sub\mfd\ of dimension $n$. 
Set $l=m-n$. If we denote by $T_{M}X$ the normal bundle of $M$ in $X$, by the tubular \nbd\ theorem and excision,
we have a canonical \iso
\[
H^{q}(X,X\ssm M;\Z)\simeq H^{q}(T_{M}X,T_{M}X\ssm M;\Z).
\]
Note that if $X$ and $M$ are orientable, the bundle  $T_{M}X$ is orientable and thus the total space is also
orientable. We endow them with orientations according to the above rule.
In this case the Thom class $\vP_{M}\in H^{l}(X,X\ssm M;\Z)$ of $M$ in $X$ is defined to be the class corresponding to the Thom class of $T_{M}X$ under the above \iso\ for $q=l$.
We also have the Thom \iso
\begin{equationth}\label{thom}
T:H^{q-l}(M;\Z)\overset\sim\lra H^{q}(X,X\ssm M;\Z).
\end{equationth}

\subsection{Relative de~Rham and relative Dolbeault cohomologies}

Let $X$ be a complex \mfd\ of dimension $n$. We consider the following two cases where 
there is a natural relation between the two cohomology theories.
\vv

\noindent
{\bf (I)} Noting that, for any $(n,q)$-form $\o$, $\bp\o=d\o$,
there is a natural morphism
\begin{equationth}\label{DdR}
H^{n,q}_{\bar\vt}(\W,\W')\lra H^{n+q}_{D}(\W,\W').
\end{equationth}

In particular, this is used to define the integration on the \v{C}ech-Dolbeault cohomology in the subsequent section.
\vv

\noindent
{\bf (II)} 
We define
$\rho^{q}:\scE^{(q)}\lra \scE^{(0,q)}$
by assigning to a $q$-form $\o$ its $(0,q)$-component $\o^{(0,q)}$. Then $\rho^{q+1}(d\o)=\bp(\rho^{q}\o)$
and we have\,:

\begin{proposition} There is a natural morphism of complexes
\[
\SelectTips{cm}{}
\xymatrix@C=.7cm
@R=.7cm
{ 0\ar[r]& \C\ar[r] \ar[d]^-{\iota}&\scE^{(0)}\ar[r]^-{d}\ar[d]^-{\rho^{0}} &\scE^{(1)}\ar[d]^-{\rho^{1}}\ar[r]^-{d}&\cdots\ar[r]^-{d}&\scE^{(q)}\ar[d]^-{\rho^{q}}\ar[r]^-{d}&\cdots\\
0\ar[r] & \scO \ar[r] & \scE^{(0,0)}\ar[r]^-{\bp}&\scE^{(0,1)}\ar[r]^-{\bp}&\cdots\ar[r]^-{\bp}&\scE^{(0,q)}\ar[r]^-{\bp}&\cdots.}
\]
\end{proposition}

\begin{corollary}\label{cordRD} There is a  natural morphism  
$\rho^{q}:H^{q}_{D}(X,X')\ra H^{0,q}_{\bar\vt}(X,X')$ that makes
the following  diagram commutative\,{\rm :}
\[
\SelectTips{cm}{}
\xymatrix@R=.4cm
@C=.6cm
{ H^{q}_{D}(X,X')\ar[r]^-{\rho^{q}} \ar@{-}[d]^-{\wr}& H^{0,q}_{\bar\vt}(X,X')\ar@{-}[d]^-{\wr}\\
 H^{q}(X,X';\C)\ar[r]^-{\iota} &H^{q}(X,X';\scO) .}
\]
\end{corollary}

Recall that we have the analytic de~Rham complex
\[
0\lra\C\overset{\iota}\lra\scO\overset{d}\lra\scO^{(1)}\overset{d}\lra\cdots\overset{d}\lra\scO^{(n)}\lra 0
\]
and we have an \iso\ of complexes (cf. Proposition \ref{propdcomp2})\,:
\begin{equationth}\label{lcrd}
\SelectTips{cm}{}
\xymatrix@C=.5cm
@R=.4cm
{ 0\ar[r]& H^{q}_{D}(X,X')\ar[r]^-{\rho^{q}} \ar@{-}[d]^-{\wr}&H^{0,q}_{\bar\vt}(X,X')\ar[r]^-{\partial}\ar@{-}[d]^-{\wr} &H^{1,q}_{\bar\vt}(X,X')\ar@{-}[d]^-{\wr}\ar[r]^-{\partial}&\cdots\ar[r]^-{\partial}&H^{n,q}_{\bar\vt}(X,X')\ar[r]\ar@{-}[d]^-{\wr}&0\\
0\ar[r] & H^{q}(X,X';\C) \ar[r]^-{\iota} & H^{q}(X,X';\scO) \ar[r]^-{d}&H^{q}(X,X';\scO^{(1)})\ar[r]^-{d}&\cdots\ar[r]^-{d}&H^{q}(X,X';\scO^{(n)})\ar[r]&0.}
\end{equationth}

Although the following appears to be well-known, we give a proof for the sake of completeness. 

\begin{theorem}\label{thexactpcd} If $H^{q}(X,X';\C)=0$ and $H^{q}(X,X';\scO^{(p)})=0$ for $p\ge 0$ and $q\ne q_{0}$, then 
\[
0\ra H^{q_{0}}(X,X';\C)\overset{\iota}\ra H^{q_{0}}(X,X';\scO)\overset{d}\ra H^{q_{0}}(X,X';\scO^{(1)})\overset{d}\ra\cdots\overset{d}\ra H^{q_{0}}(X,X';\scO^{(n)})\ra 0
\]
is exact.
\end{theorem}

\begin{proof} 
Let $(\scF^{\bullet,\bullet}, d_{1}, d_{2})$ be a double complex of flabby sheaves \st, in the following diagram, each row
is exact, each column is a flabby resolution and the diagram consisting of the first and second rows is commutative (note
that $d_{1}\circ d_{2}+d_{2}\circ d_{1}=0$).


\[
\SelectTips{cm}{}
\xymatrix@R=.6cm
@C=.8cm{{}& 0\ar[d] &0 \ar[d] & 0\ar[d]& {} &0\ar[d]\\
0 \ar[r] & \C\ar[r]^-{\iota} \ar[d]& \scO \ar[r]^-{d}\ar[d]& \scO^{(1)}\ar[r]^-{d}\ar[d]& \cdots\ar[r]^-{d} &\scO^{(n)}\ar[r]\ar[d]&0\\
0 \ar[r] & \scF^{-1,0} \ar[r]^-{d_{1}}\ar[d]^-{d_{2}}& \scF^{0,0}\ar[r]^-{d_{1}} \ar[d]^-{d_{2}}& 
\scF^{1,0} \ar[r]^-{d_{1}}\ar[d]^-{d_{2}}&\cdots\ar[r]^-{d_{1}} &\scF^{n,0}\ar[r]\ar[d]^-{d_{2}}&0\\
0 \ar[r] & \scF^{-1,1} \ar[r]^-{d_{1}}\ar[d]^-{d_{2}}& \scF^{0,1}\ar[r]^-{d_{1}} \ar[d]^-{d_{2}}& \scF^{1,1}\ar[r]^-{d_{1}}\ar[d]^-{d_{2}}&\cdots\ar[r]^-{d_{1}} &\scF^{n,1}\ar[r]\ar[d]^-{d_{2}}& 0\\
{} &\vdots &\vdots&\vdots&{}&\vdots.}
\]
We have the associated double complex 
$(F^{\bullet,\bullet}, d_{1}, d_{2})$, $F^{p,q}=\scF^{p,q}(X,X')$. 
Denoting by $F^{\bullet}$ the single complex associated with $F^{\bullet,\bullet}$, consider the first
spectral sequence
\[
'\hspace{-.6mm}E^{p,q}_{2}=H^{p}_{d_{1}}H^{q}_{d_{2}}(F^{\bullet,\bullet})\Longrightarrow H^{p+q}(F^{\bullet}).
\]
By assumption $H^{q}_{d_{2}}(F^{p,\bullet})=0$ for $p\ge -1$ and $q\ne q_{0}$ and $H^{q_{0}}_{d_{2}}(F^{p,\bullet})=H^{q_{0}}(X,X';\scO^{(p)})$. Thus $H^{p}_{d_{1}}H^{q_{0}}(X,X';\scO^{(\bullet}))\simeq H^{p+q_{0}}(F^{\bullet})$.
On the other hand, in the second spectral sequence
\[
''\hspace{-.6mm}E^{q,p}_{2}=H^{q}_{d_{2}}H^{p}_{d_{1}}(F^{\bullet,\bullet})\Longrightarrow H^{p+q}(F^{\bullet}),
\]
$H^{p}_{d_{1}}(F^{\bullet,q})=0$, for $p\ge -1$ and $q\ge 0$, so that $H^{r}(F^{\bullet})=0$ for all $r$.
\end{proof}

As an application, combining with (\ref{lcrd}), we have the de~Rham complex for ``hyperforms'' (cf. \eqref{hypdR} below and \cite{HIS}). 


\lsection{Cup product and integration}\label{seccupint}

Let $X$ be a complex \mfd\ of dimension $n$ and $\W=\{W_{\a}\}_{\a\in I}$ a covering of $X$.
\subsection{Cup product}

We have the complex  $\scE^{(p,\bullet)}(\W)$ as considered in Subsection \ref{ssecCD}. We define the ``cup product"
\begin{equationth}\label{cupgen}
\scE^{(p,q)}(\W)\times \scE^{(p',q')}(\W)\lra \scE^{(p+p',q+q')}(\W)
\end{equationth}
by assigning to $\xi$ in $\scE^{(p,q)}(\W)$ and $\eta$ in $\scE^{(p',q')}(\W)$ the cochain
$\xi\smallsmile\eta$ in $\scE^{(p+p',q+q')}(\W)$ given by
\[
(\xi\smallsmile\eta)_{\a_0\dots\a_r}
=\sum_{\nu=0}^r(-1)^{(p+q-\nu)(r-\nu)}\xi_{\a_0\dots\a_{\nu}}\wedge\eta_{\a_{\nu}\dots\a_r}.
\]
Then $\xi\smallsmile\eta$ is bilinear in $(\xi,\eta)$ and we have
\begin{equationth}\label{cupCDcochain}
\bar\vt(\xi\smallsmile\eta)=\bar\vt\xi\smallsmile\eta+(-1)^{p+q}\xi\smallsmile\bar\vt\eta.
\end{equationth}
Thus it induces the cup product
\begin{equationth}\label{cupCD}
H^{p,q}_{\bar\vt}(\W)\times H^{p',q'}_{\bar\vt}(\W)\lra 
H^{p+p',q+q'}_{\bar\vt}(\W)
\end{equationth}
compatible, via the \iso\ of Theorem \ref{thsummary}.\,1,
with the  product in the
Dolbeault cohomology induced by the exterior product of forms. 

If $\W'$ is a subcovering of $\W$, the cup product \eqref{cupgen} induces
\[
\scE^{(p,q)}(\W,\W')\times \scE^{(p',q')}(\W)\lra \scE^{(p+p',q+q')}(\W,\W')
\]
which in turn induces the cup product
\[
H^{p,q}_{\bar\vt}(\W,\W')\times H^{p',q'}_{\bar\vt}(\W)\lra 
H^{p+p',q+q'}_{\bar\vt}(\W,\W').
\]

In the case of a covering $\V=\{V_{0},V_{1}\}$ with two open sets, the cup product
\begin{equationth}\label{cup}
\scE^{(p,q)}(\V)\times \scE^{(p',q')}(\V)\lra \scE^{(p+p',q+q')}(\V),
\end{equationth}%
assigs to $\xi$ in $\scE^{(p,q)}(\V)$ and $\eta$ in $\scE^{(p',q')}(\V)$ the cochain
$\xi\smallsmile\eta$ in $\scE^{(p+p',q+q')}(\V)$ given by
\[
\begin{aligned}
&(\xi\smallsmile\eta)_0=\xi_0\wedge\eta_0,\quad(\xi\smallsmile\eta)_1=\xi_1\wedge\eta_1\quad\text{and}\\
&(\xi\smallsmile\eta)_{01}=(-1)^{p+q}\xi_0\wedge\eta_{01}+\xi_{01}\wedge\eta_1.
\end{aligned}
\]

Suppose $S$ is a closed set in $X$. Let $V_{0}=X\ssm S$ and $V_{1}$ a \nbd\ of $S$ and consider the
covering $\V=\{V_{0},V_{1}\}$. Then we see that \eqref{cup} induces a pairing
\begin{equationth}\label{cuprelcoch}
\scE^{(p,q)}(\V,V_{0})\times \scE^{(p',q')}(V_{1})\lra \scE^{(p+p',q+q')}(\V,V_{0}),
\end{equationth}%
assigning to $\xi=(\xi_{1},\xi_{01})$  and $\eta_{1}$ the cochain
$(\xi_1\wedge\eta_1,\xi_{01}\wedge\eta_1)$.  It induces the pairing
\begin{equationth}\label{cuprel}
H^{p,q}_{\bar\vt}(X,X\ssm S)\times H^{p',q'}_{\bp}(V_{1})\lra H^{p+p',q+q'}_{\bar\vt}(X,X\ssm S).
\end{equationth}

More generally, let $S_{i}$ be a closed set in $X$,  $i=1,2$. Let $V_{0}^{(i)}=X\ssm S_{i}$ and $V_{1}^{(i)}$ a \nbd\ of $S_{i}$ and consider the
covering $\V^{(i)}=\{V_{0}^{(i)},V_{1}^{(i)}\}$ of $X$. We set $S=S_{1}\cap S_{2}$, $V_{0}=X\ssm S$ and 
$V_{1}$ an open \nbd\ of $S$ contained in $V_{1}^{(1)}\cap V_{1}^{(2)}$ and consider the covering 
$\V=\{V_{0},V_{1}\}$ of $X$.
The set $V_{0}$
 is covered by two open sets $V_{0}^{(1)}$ and $V_{0}^{(2)}$.
Let $\{\rho_{1},\rho_{2}\}$ be a partition of unity subordinate to the covering.
We define a paring
\begin{equationth}\label{cuprelcochgen}
\scE^{(p,q)}(\V^{(1)},V_{0}^{(1)})\times \scE^{(p',q')}(\V^{(2)},V_{0}^{(2)})\lra \scE^{(p+p',q+q')}(\V,V_{0})
\end{equationth}
by
\[
(\xi_{1},\xi_{01})\ssl (\eta_{1},\eta_{01})
=(\xi_{1}\wedge\eta_{1},\rho_{1}\xi_{01}\wedge\eta_{1}+(-1)^{p+q}(\rho_{2}\xi_{1}\wedge\eta_{01}-\bp\rho_{1}\wedge\xi_{01}\wedge\eta_{01}).
\]
Then we see that the equality \eqref{cupCDcochain} also holds and we have the product
\begin{equationth}\label{cuprelgen}
H^{p,q}_{\bar\vt}(X,X\ssm S_{1})\times H^{p',q'}_{\bar\vt}(X,X\ssm S_{2})\lra H^{p+p',q+q'}_{\bar\vt}(X,X\ssm S).
\end{equationth}
It is not difficule to see that (\ref{cuprelgen}) does not depend on the choice of the partition of unity 
$\{\rho_{1},\rho_{2}\}$.

In particular, if $S_{2}=X$, we may set $\rho_{1}\equiv 1$ and $\rho_{2}\equiv 0$ and (\ref{cuprelcochgen})
reduces to (\ref{cuprelcoch}).

The above may be used to define, for two pairs $(X,S)$ and $(Y,T)$ the product
\[
H^{p,q}_{\bar\vt}(X,X\ssm S)\times H^{p',q'}_{\bar\vt}(Y,Y\ssm T)\lra H^{p+p',q+q'}_{\bar\vt}(X\times Y,X\times Y\ssm S\times T).
\]

\subsection{Integration}

Recall that there is a natural morphism
$H^{n,q}_{\bar\vt}(\W)\ra H^{n+q}_{D}(\W)$ (cf. (\ref{DdR})).
Thus the integration
 on $H^{2n}_{D}(\W)$ carries directly on to  $H^{n,n}_{\bar\vt}(\W)$.
We briefly recall the integration theory on the \v{C}ech-de Rham cohomology and refer to \cite{Leh2} and \cite{Su2} for details.

Let $X$ be a $C^{\infty}$ \mfd\ of dimension $m$
and $\W=\{W_{\a}\}_{\a\in I}$  a covering of $X$. 
We assume that  $I$ is an ordered 
set such that if $W_{\a_0\dots\a_q}\ne\emptyset$, the induced order on the subset
$\{\a_0,\dots,\a_q\}$ is total. We set
\[
I^{(q)}=\{\,(\a_0,\dots,\a_q)\in I^{q+1}\mid \a_0<\cdots<\a_q\,\}.
\]

A {\em honeycomb system} adapted to $\W$ is a collection $\{R_\a\}_{\a\in I}$ 
\st
\begin{enumerate}
\item[(1)] each $R_\a$ is an $m$-dimensional \mfd\ with piecewise $C^{\infty}$ boundary in $W_\a$ and $X=\bigcup_\a R_\a$,
\item[(2)]  if $\a\ne\beta$, then $\op{Int}R_\a\cap\op{Int}R_\beta=\emptyset$,
\item[(3)]  if $W_{\a_0\dots \a_q}\ne\emptyset$ and if the $\a_\nu$'s are distinct, then $R_{\a_0\dots \a_q}$
 is an $(m-q)$-dimensional \mfd\ with piecewise $C^{\infty}$  boundary,
\item[(4)]  if the set $\{\a_0,\dots,\a_q\}$ is maximal, then $R_{\a_0\dots \a_q}$ has no boundary.
\end{enumerate}

In the above, we denote by $\op{Int}R$  the interior of a subset $R$ in $X$ and we set 
$R_{\a_0\dots \a_q}=\bigcap_{\nu=0}^q R_{\a_\nu}$,
which is equal to $\bigcap_{\nu=0}^p\partial R_{\a_\nu}$ by (2) above.  Also, $\{\a_0,\dots,\a_q\}$
being maximal means that if $W_{\a,\a_0,\dots,\a_q}\ne\emptyset$, then $\a$ is in $\{\a_0,\dots,\a_q\}$.


Suppose $X$ is oriented. Let $R$ be an $m$-dimensional \mfd\ with $C^{\infty}$ boundary $\partial R$ in $X$. Then 
$R$ is oriented so that it has the same orientation as $X$. In this case, the  boundary $\partial R$ is orientable and is oriented as follows.
Let $p$ be a point in $\partial R$. There exist a \nbd\ $U$ of $p$ and a $C^{\infty}$ coordinate system $(x_{1},\dots,x_{m})$
on $U$ \st\ $R\cap U=\{\,x\in U\mid x_{1}\le 0\,\}$. We orient $\partial R$ so that if $(x_{1},\dots,x_{m})$ is
a positive coordinate system on $X$, $(x_{2},\dots,x_{m})$ is a positive coordinate system on $\partial R$.
A \mfd\ with piecewise $C^{\infty}$  boundary is oriented similarly.

If $\{R_{\a}\}$ is a honeycomb system, we orient $R_{\a_0\dots \a_q}$ by the following rules\,:
\begin{enumerate}

\item[(1)] each $R_\a$  and its boundary are oriented as above,
\item[(2)] for $(\a_0,\dots,\a_q)$ in $I^{(q)}$, $q\ge 1$, $R_{\a_0\dots \a_q}$ is oriented as a part of 
$\partial R_{\a_0\dots \a_{q-1}}$,
\item[(3)]  for  $(\a_0,\dots,\a_q)$ in $I^{q+1}$, we set
\[
R_{\a_0\dots\a_q}=\begin{cases} \sgn\rho\cdot R_{\a_{\rho(0)}\dots\a_{\rho(p)}}&\quad\text{if}\ W_{\a_0\dots\a_q}\ne\emptyset\ \text{and the}\ \a_i\text{'s are distnct},\\
\emptyset&\quad\text{otherwise},
\end{cases}
\]
where $\rho$ is the permutation \st\ $\a_{\rho(0)}<\cdots <\a_{\rho(q)}$.  
\end{enumerate}
With the above convention, we may write\,:
\begin{equationth}\label{bdryhoney}
\partial R_{\a_0\dots \a_q}=\sum_{\a\in I} R_{\a_0\dots \a_q\a}.
\end{equationth}

Suppose moreover that $X$ is compact, then each $R_\a$ is 
compact and we may define the integration
\[
\int_X:\scE^{(m)}(\W)\lra \C
\]
by setting
\[
\int_X\xi=\sum_{q=0}^m\Big(\sum_{(\a_0,\dots,\a_q)\in
I^{(q)}}\int_{R_{\a_0\dots\a_q}}\xi_{\a_0\dots\a_q}\Big)\qquad\text{for}\ \ \xi\in \scE^{(m)}(\W).
\]
Then  we see that 
it induces the integration on the cohomology 
\[
\int_X:H^{m}_{D}(\W)\lra \C,
\]
which is compatible
with the usual
integration on the de~Rham cohomology $H^{m}_{d}(X)$. 

\

Now let $X$ be a complex \mfd\ of dimension $n$ and $\W$ a covering of $X$. As a real \mfd, $X$ is always orientable and we specify an orientation in the sequel. However we note that the orientation we consider is not necessarily the ``usual
one''. Here we say an orientation of $X$ is usual if $(x_{1},y_{1},\dots,x_{n},y_{n})$ is a positive coordinate
system when  $(z_{1},\dots,z_{n})$, $z_{i}=x_{i}+\sqrt{-1}y_{i}$, is a coordinate system on $X$.

Using the natural morphism $H^{n,q}_{\bar\vt}(\W)\ra H^{n+q}_{D}(\W)$,
if $X$ is compact,
we may define the integration on $H^{n,n}_{\bar\vt}(\W)$ as the composition
\begin{equationth}\label{intCD}
H^{n,n}_{\bar\vt}(\W)\lra H^{2n}_{D}(\W)\overset{\int_{X}}\lra\C.
\end{equationth}

Let $K$ be a
compact  set in $X$ ($X$ may not be compact). Letting $V_{0}=X\ssm K$ and $V_{1}$ a \nbd\ of $K$, we consider the coverings 
$\V=\{V_{0},V_{1}\}$ and $\V'=\{V_{0}\}$. Let $\{R_0,R_1\}$ be a  honeycomb system adapted to $\V$. 
In this case we may take as $R_{1}$ a compact $2n$-dimensional  \mfd\ with $C^{\infty}$ boundary  in $V_{1}$ containing $K$
in its interior and set $R_{0}=X\ssm \op{Int}R_{1}$.
Then $R_{01}=-\partial R_{1}$ (cf. \eqref{bdryhoney}) and we 
have the integration
on $\scE^{(n,n)}(\V,\V')$ given by, for $\xi=(\xi_{1},\xi_{01})$,
\[
\int_X\xi=\int_{R_1}\xi_1+\int_{R_{01}}\xi_{01}.
\]
This again induces the integration on the cohomology
\begin{equationth}\label{intDrel}
\int_X:H^{n,n}_{\bar\vt}(X,X\ssm K)\lra \C.
\end{equationth}

\lsection{Local duality morphism}\label{secld}


Let $X$ be a complex \mfd\ of dimension $n$.

First, if $X$ is compact, the bilinear pairing
\[
 H^{p,q}_{\bar\vt}(\W)\times H^{n-p,n-q}_{\bar\vt}(\W)\overset{\smallsmile}\lra H^{n,n}_{\bar\vt}(\W)
 \overset{\int_{X}}\lra \C
\]
given
as the composition of the cup product \eqref{cupCD} and the integration \eqref{intCD} induces the Kodaira-Serre duality
\begin{equationth}\label{3.4}
KS_{X}:H^{p,q}_{\bp}(X)\simeq H^{p,q}_{\bar\vt}(\W)\overset{\sim}{\ra}
H^{n-p,n-q}_{\bar\vt}(\W)^*\simeq H^{n-p,n-q}_{\bp}(X)^*,
\end{equationth}
where, for a complex vector space $\sV$, $\sV^{*}$ denotes its algebraic dual.

Now we consider the case where $X$ may not be compact. Let $\scE^{(p,q)}_{c}(X)$ denote the space of $(p,q)$-forms
with compact support on $X$. The $q$-th cohomology of the complex $(\scE^{(p,\bullet)}_{c}(X),\bp)$ will be denoted
by $H^{p,q}_{\bp, c}(X)$.
The bilinear pairing
\[
 \scE^{(p,q)}(X)\times \scE^{(n-p,n-q)}_{c}(X)\overset\wedge\lra \scE^{(n,n)}_{c}(X)\overset{\int_{X}}\lra \C
\]
induces the Serre morphism
\[
S_{X}:H^{p,q}_{\bp}(X)\lra H^{n-p,n-q}_{\bp, c}(X)^{*}.
\]
Let $K$ be a compact set in $X$.
The cup product \eqref{cuprel}
followed by the integration (\ref{intDrel}) gives a bilinear pairing
\[
H^{p,q}_{\bar\vt}(X,X\ssm K)\times H^{n-p,n-q}_{\bp}(V_1)\overset{\smallsmile}\lra H^{p,q}_{\bar\vt}(X,X\ssm K)\overset{\int_{X}}\lra\C.
\]
Setting
\[
H^{n-p,n-q}_{\bp}[K]
=\lim_{\underset{V_{1}\supset K}\lra}H^{n-p,n-q}_{\bp}(V_1),
\]
where $V_{1}$ runs through open \nbd s of $K$,
this induces a morphism
\begin{equationth}\label{3.5}
\bar A_{X,K}:H^{p,q}_{\bar\vt}(X,X\ssm K)\lra H^{n-p,n-q}_{\bp}[K]^*
\end{equationth}
which we call the {\em complex analytic Alexander morphism}, or the {\em $\bp$-Alexander morphism} for short.
We have the following commutative diagram\,:
\[
\SelectTips{cm}{}
\xymatrix
{H^{p.q}_{\bar\vt}(X,X\ssm K)\ar[r]^-{j^*}\ar[d]^{\bar A_{X,K}} &H^{p.q}_{\bp}(X)\ar[d]^{S_{X}}\\
H^{n-p,n-q}_{\bp}[K]^{*} \ar[r]^-{j_*}&\ H^{n-p,n-q}_{\bp,c}(X)^{*},}
\]
which will be extended to a commutative diagram of long exact sequences (cf. Theorem~\ref{thlongdual} below).

\paragraph{An exact sequence\,:}

Let $S$ be a closed set in $X$. Since a differential form on $X\ssm S$ with compact support may naturally be
 thought of as a form on $X$ with compact support, there is a natural morphism
\[
i^{*}:H^{p,q}_{\bp,c}(X\ssm S)\lra H^{p,q}_{\bp,c}(X).
\]

We also have a natural morphism
\[
j^{*}:H^{p,q}_{\bp,c}(X)\lra H^{p,q}_{\bp}[S]=\lim_{\underset{V_{1}\supset S}\lra} H^{p,q}_{\bp}(V_{1})
\]
as the composition 
\[
H^{p,q}_{\bp,c}(X)\lra H^{p,q}_{\bp}(X)\lra H^{p,q}_{\bp}[S].
\]

Let $K$ be a compact set in $X$. 
We define a morphism
\begin{equationth}\label{homgam}
\gamma^{*}:H^{p,q}_{\bp}[K]\lra H^{p,q+1}_{\bp,c}(X\ssm K)
\end{equationth}
as follows. Take an element $a$ in $H^{p,q}_{\bp}[K]$. Then it is represented by $[\eta]$ in $H^{p,q}_{\bp}(V_{1})$
for some \nbd\ $V_{1}$ of $K$, which may be assumed to be relatively compact. Let $V_{0}=X\ssm K$ and consider the 
covering $\V=\{V_{0},V_{1}\}$ of $X$. Let $\{\rho_{0},\rho_{1}\}$ be a $C^{\infty}$ partition of unity subordinate to $\V$. Then, noting that the support of  $\bp\rho_{1}$ is in $V_{01}=V_{1}\ssm K$, we see that
$\eta\wedge\bp\rho_{1}$ is a $\bp$-closed $(p,q+1)$-form with
compact support in $X\ssm K$. 

\begin{lemma} The class of $\eta\wedge\bp\rho_{1}$ in $H^{p,q+1}_{\bp,c}(X\ssm K)$ is uniquely determined by
$a$.
\end{lemma}
\begin{proof}
Suppose $a$ is also  represented by $[\eta']$ in $H^{p,q}_{\bp}(V'_{1})$, $V'_{1}\subset V_{1}$. Then there exists a $(p,q-1)$-form $\xi$
on $V'_{1}$ \st\ $\eta'-\eta=\bp\xi$ on $V'_{1}$. We compute
\[
\eta'\wedge\bp\rho'_{1}-\eta\wedge\bp\rho_{1}=(\eta'-\eta)\wedge\bp\rho'_{1}+\eta\wedge(\bp\rho'_{1}-\bp\rho_{1}).
\]
The form $\xi\wedge\bp\rho'_{1}$ is a $(p,q)$-form with compact support in $X\ssm K$ and $(\eta'-\eta)\wedge\bp\rho'_{1}=\bp(\xi\wedge\bp\rho'_{1})$. Also the support of $\rho'_{1}-\rho_{1}$ is in $V_{01}$
and $\eta\wedge(\bp\rho'_{1}-\bp\rho_{1})=(-1)^{p+q}\bp((\rho'_{1}-\rho_{1})\eta)$. 
\end{proof}

Thus we define the morphism (\ref{homgam}) by $\gamma^{*}(a)=[\eta\wedge\bp\rho_{1}]$.

\begin{theorem}\label{thexactloc} The following sequence is exact\,{\rm :}
\[
\cdots\lra  H^{p,q}_{\bp,c}(X\ssm K)\overset{i^{*}}\lra H^{p,q}_{\bp,c}(X)\overset{j^{*}}\lra H^{p,q}_{\bp}[K]\overset{\gamma^{*}}\lra H^{p,q+1}_{\bp,c}(X\ssm K)
\overset{i^{*}}\lra\cdots.
\]
\end{theorem}
\begin{proof} To show $\op{Im}i^{*}\subset\op{Ker}j^{*}$, let $[\o_{0}]$ be a class in $H^{p,q}_{\bp,c}(X\ssm K)$. Take $V_{1}$ so that it avoids the support of $\o_{0}$. Then the class $i^{*}[\o_{0}]$ is mapped to zero
by $H^{p,q}_{\bp,c}(X)\ra H^{p,q}_{\bp}(X)\ra H^{p,q}_{\bp}(V_{1})$. To show $\op{Ker}j^{*}\subset\op{Im}i^{*}$,
let $[\o]$ be a class in $H^{p,q}_{\bp,c}(X)$ \st\ $j^{*}([\o])=0$. Then there exist a \nbd\ $V_{1}$ of $K$ and 
a $(p,q-1)$-form $\t$ on $V_{1}$ \st\ $\o=\bp\t$ on $V_{1}$. The form $\rho_{1}\t$ is a $(p,q-1)$-form on $X$
and $\o'=\o-\bp(\rho_{1}\t)$ is a $\bp$-closed $(p,q)$-form with compact support in $X\ssm K$. Thus $[\o]=i^{*}[\o']$.

To show $\op{Im}j^{*}\subset\op{Ker}\gamma^{*}$, let $[\o]$ be a class in $H^{p,q}_{\bp,c}(X)$. Then $\gamma^{*}j^{*}[\o]$ is by definition, the class of $\o\wedge\bp\rho_{1}$. The form $\rho_{0}\o$ is a $(p,q)$-form with
compact support in $X\ssm K$ and $\o\wedge\bp\rho_{1}=(-1)^{p+q+1}\bp(\rho_{0}\o)$. To show $\op{Ker}\gamma^{*}\subset\op{Im}j^{*}$, let $a$ be a class in $H^{p,q}_{\bp}[K]$ represented by $(V_{1},\eta)$. If $\gamma^{*}a=0$, there
exists a $(p,q)$-form $\xi$ on $X\ssm K$ with compact support \st\ $\eta\wedge\bp\rho_{1}=\bp\xi$ on $X\ssm K$.
We may think of $\xi$ as a $(p,q)$-form on $X$ and the equality holds on $X$, the both sides being $0$ near $K$. The form $\rho_{1}\eta$ is a $(p,q)$-form on $X$ and 
$\o=\xi+(-1)^{p+q+1}\rho_{1}\eta$ is
a $(p,q)$-form on $X$ extending $\eta$ (restricted to a \nbd\ of $K$) and we have 
$\bp\o=\bp\xi-\eta\wedge\bp\rho_{1}=0$.

To show $\op{Im}\gamma^{*}\subset\op{Ker}i^{*}$, let $a$ be a class in $H^{p,q}_{\bp}[K]$ represented by $(V_{1},\eta)$. By definition, $\gamma^{*}a$ is the class of $\eta\wedge\bp\rho_{1}$, which may be written as
$(-1)^{p+q}\bp(\rho_{1}\eta)$ and $\rho_{1}\eta$ is a $(p,q)$-form on $X$ with compact support.
To show $\op{Ker}i^{*}\subset\op{Im}\gamma^{*}$, let $[\a]$ be a class in $H^{p,q+1}_{\bp,c}(X\ssm K)$.
If $i^{*}[\a]=0$, there exists a $(p,q)$-form $\b$ with compact support in $X$ \st\ $\a=\bp\b$. Take $\rho_{0}$ so that $\rho_{0}\equiv 1$ on the support of $\a$. Then  $\bp(\rho_{0}\b)=\bp\rho_{0}\wedge\b+\rho_{0}\bp\b$,
$\rho_{0}\b$ is a $(p,q)$-form with compact support in $X\ssm K$ and $\rho_{0}\bp\b=\a$. Thus 
$\a=(-1)^{p+q}\b\wedge\bp\rho_{1}+\bp(\rho_{0}\b)$.
\end{proof}

\begin{remark} The above is an expression of the following exact sequence (cf. \cite{G}, \cite{K}) in our framework\,:
\[
\cdots\lra  H^{q}_{c}(X\ssm K;\scO^{(p)})\overset{i^{*}}\lra H^{q}_{c}(X;\scO^{(p)})\overset{j^{*}}\lra H^{q}(K;\scO^{(p)})\overset{\gamma^{*}}\lra H^{q+1}_{c}(X\ssm K;\scO^{(p)})
\overset{i^{*}}\lra\cdots.
\]
\end{remark}

\begin{theorem}\label{thlongdual} In the above situation, we have a commutative diagram with exact rows\,{\rm :}
\[
\SelectTips{cm}{}
\xymatrix@C=.26cm
@R=.7cm{\cdots\ar[r]& H^{p,q-1}_{\bp}(X \ssm K)\ar[r]^-{\delta} \ar[d]^{S_{X\ssm K}}&H^{p,q}_{\bar\vt}(X,X\ssm K)\ar[r]^-{j^*}\ar[d]^{\bar A} &H^{p,q}_{\bp}(X)\ar[d]^{S_{X}}\ar[r]^-{i^{*}}&H^{p,q}_{\bp}(X \ssm K)\ar[r]\ar[d]^{S_{X\ssm K}}&\cdots\\
\cdots\ar[r] & H^{n-p,n-q+1}_{\bp,c}(X\ssm K)^{*} \ar[r]^-{\gamma_*} & H^{n-p,n-q}_{\bp}[K]^{*} \ar[r]^-{j_*}&\ H^{n-p,n-q}_{\bp,c}(X)^{*}\ar[r]^-{i_{*}}&H^{n-p,n-q}_{\bp,c}(X\ssm K)^{*}\ar[r]&\cdots,}
\]
where the second row is the sequence dual to the one in Theorem \ref{thexactloc}.
\end{theorem}

\begin{proof} For the first rectangle, it amounts to showing that
\[
\int_{X\ssm K}\t\wedge\eta\wedge\bp\rho_{1}=-\int_{R_{01}}\t\wedge\eta,
\]
where $\t$ is a $\bp$-closed $(p,q-1)$-form on $X\ssm K$, $\eta$ is a $\bp$-closed $(n-p,n-q)$-form on $V_{1}$,
$\{R_{0},R_{1}\}$ is a honeycomb system adapted to $\V$ and
$\{\rho_{0},\rho_{1}\}$ is a partition of unity subordinate to $\V$. We may assume that $\rho_{1}\equiv 1$ on $R_{1}$, thus in particular the support of $\bp\rho_{1}$ is in $R_{0}$. Since $\rho_{1}\eta$ is a $(n-p,n-q)$-form on $X$,
$\t\wedge\rho_{1}\eta$ is an $(n,n-1)$-form on $X\ssm K$. Since it is $\partial$-closed, we have
\[
d(\t\wedge\rho_{1}\eta)=\bp(\t\wedge\rho_{1}\eta)=-\t\wedge\eta\wedge\bp\rho_{1}
\]
and by the Stokes theorem
\[
\int_{R_{0}}\t\wedge\eta\wedge\bp\rho_{1}=-\int_{R_{01}}\t\wedge\rho_{1}\eta
=-\int_{R_{01}}\t\wedge\eta.
\]

For the second rectangle, it amounts to showing that
\begin{equationth}\label{eq2}
\int_{X}(\rho_{1}\sigma_{1}-\bp\rho_{0}\wedge\sigma_{01})\wedge\o
=\int_{R_{1}}\sigma_{1}\wedge\o+\int_{R_{01}}\sigma_{01}\wedge\o,
\end{equationth}
where $(\sigma_{1},\sigma_{01})$ is a cocycle representing a class in $H^{p.q}_{\bar\vt}(X,X\ssm K)$ and $\o$
a $\bp$-closed $(n-p,n-q)$-form on $X$ with compact support. We take $\{\rho_{0}, \rho_{1}\}$ so that
the support of $\rho_{1}$ is contained in $R_{1}$. Then the left hand side is the integral on $R_{1}$ and is written as
\[
\int_{R_{1}}\sigma_{1}\wedge\o-\int_{R_{1}}(\rho_{0}\sigma_{1}+\bp\rho_{0}\wedge\sigma_{01})\wedge\o.
\]
The form
$\rho_{0}\sigma_{01}$ is defined on $V_{1}$ and $\rho_{0}\sigma_{01}\wedge\o$ is an $(n,n-1)$-form defined
on $V_{1}$. We have
\[
d(\rho_{0}\sigma_{01}\wedge\o)=\bp(\rho_{0}\sigma_{01}\wedge\o)=(\rho_{0}\sigma_{1}+\bp\rho_{0}\wedge\sigma_{01})\wedge\o
\]
and we have (\ref{eq2}) by the Stokes theorem.

The commutativity of the third rectangle follows directly from the definition.
\end{proof}

An interesting problem would be to see when  $\bar A$ is an \iso. For this, we need to consider topological duals instead of 
algebraic duals and we briefly recall the theory of topological vector spaces and the Serre duality
(cf. \cite{K}, \cite{Se}, \cite{T}).  
In the sequel, for a locally convex topological vector space $\sV$, we denote 
by $\sV'$ its strong topological dual.

A Fr\'echet-Schwartz space, an FS space for short,  is a locally convex space $\sV$ that can be expressed as the inverse limit 
$\sV=\underset\lla\lim\sV_{i}$ of a descending
sequence of Banach spaces $(\sV_{i},u_{i,i+1})$ with each $u_{i,i+1}:\sV_{i+1}\ra \sV_{i}$ a compact linear map.  
A closed subspace $\sW$ of
an FS space $\sV$ is FS. 
The quotient $\sV/\sW$ is also FS.
A dual Fr\'echet-Schwartz space, a DFS space for short,  is a locally convex space $\sV$ that can be expressed as the direct limit 
$\sV=\underset\lra\lim\sV_{i}$ of an ascending
sequence of Banach spaces $(\sV_{i},u_{i+1,i})$ with each $u_{i+1,i}:\sV_{i}\ra \sV_{i+1}$ an injective compact linear map. A closed subspace $\sW$ of
a DFS space $\sV$ is DFS. The quotient $\sV/\sW$ is also DFS.

If $\sV=\underset\lla\lim \sV_{i}$ is   FS,
$\sV'$  is a DFS space, which may be written as $\sV'=\underset\lra\lim \sV_{i}'$.
Also, if  $\sV=\underset\lra\lim \sV_{i}$ is  DFS,
 $X'$ is an FS space, which may be written as $\sV'=\underset\lla\lim \sV_{i}'$.
In either case we have $(\sV')'=\sV$.
Let $T:\sV_{1}\ra \sV_{2}$ be a continuous linear map of FS spaces and 
${}^{t}T:\sV_{2}'\ra \sV_{1}'$ its transpose, which is continuous.
 In this situation, $\op{Im} T$ is closed \iff\ $\op{Im} {}^{t}T$ is (the closed range theorem). 
 Note that $\op{Ker}T$ and $\op{Ker}{}^{t}T$ are always closed.
 
 Let
\[
\sV_{1}\overset{T}\lra \sV_{2}\overset{S}\lra \sV_{3}
\]
be a sequence of continuous linear maps of FS spaces \st\ $S\circ T=0$. We set $H=\op{Ker} S/\op{Im} T$
and $H^{D}=\op{Ker} {}^{t}T/\op{Im} {}^{t}\hspace{-.5mm}S$.




\begin{lemma}[Serre]  In the above situation, suppose $\op{Im} T$ and $\op{Im} S$ are closed so that $H$ is FS and $H^{D}$ is DFS.  
In this case, 
$H^{D}$ is isomorphic with $H'$.
\label{lserre}
\end{lemma}

 


The space $\scE^{(p,q)}(X)$ has a natural \str\ of
FS space. In the sequence
\begin{equationth}\label{seqdb}
\scE^{(p,q-1)}(X)\overset{\bp^{p,q-1}}\lra \scE^{(p,q)}(X)\overset{\bp^{p,q}}\lra \scE^{(p,q+1)}(X),
\end{equationth}
$\op{Ker}\bp^{p,q}$ is always closed. Thus if $\op{Im}\bp^{p,q-1}$ is closed, $H_{\bp}^{p,q}(X)$ has a natural
\str\ of FS space.
The strong dual $\scE^{(p,q)}(X)'$ of $\scE^{(p,q)}(X)$  is isomorphic  with the space $\scD_{c}^{(n-p,n-q)}(X)$ of the DFS space of $(n-p,n-q)$-currents with compact support in $X$, the \iso\ is given by assigning to $T$ in $\scD_{c}^{(n-p,n-q)}(X)$ the
linear functional 
\[
\o\mapsto\int_{X}\o\wedge T.
\]
on $\scE^{(p,q)}(X)$ (cf. \cite{K}, \cite{Se}).  The transpose of (\ref{seqdb}) is isomorphic with
\[
\scD_{c}^{(n-p,n-q+1)}(X)\overset{(-1)^{p+q}\bp}\lla \scD_{c}^{(n-p,n-q)}(X)\overset{(-1)^{p+q+1}\bp}\lla \scD_{c}^{(n-p,n-q-1)}(X).
\]
Thus if $\op{Im}\bp^{p,q}$ in (\ref{seqdb}) is closed, $H^{n-q}(\scD_{c}^{(n-p,\bullet)}(X))$ has a natural \str\ of DSF space.
By Lemma \ref{lserre}, we have\,:

\begin{theorem}[Serre] If both $\op{Im}\bp^{p,q-1}$ and $\op{Im}\bp^{p,q}$ in {\rm \eqref{seqdb}} are closed, 
there is a natural \iso.\:
\[
H^{n-q}(\scD_{c}^{(n-p,\bullet)}(X))\simeq H^{p,q}_{\bp}(X)'.
\]\label{thserre}
\end{theorem}

Note that there is a natural \iso\ $H^{q}(\scD_{c}^{(p,\bullet)}(X))\simeq H^{p,q}_{\bp,c}(X)$ and we may endow
the latter with the DFS \str\ so that the \iso\ is an \iso\ as topological vector space. With this,
under the assumption of the above theorem, we have the Serre duality 
\begin{equationth}\label{sduality}
H^{p,q}_{\bp}(X)\simeq H^{n-p,n-q}_{\bp,c}(X)'.
\end{equationth}

By a lemma of L. Schwartz, if $\dim H^{p,q}_{\bp}(X)$ is finite, $\op{Im}\bp^{p,q-1}$ is closed (cf. \cite{K}, \cite{Se}).
Thus, if $X$ is compact, (\ref{sduality}) reduces to (\ref{3.4}) for all $p$ and $q$. In the case $X$ is Stein, we have
$H^{p,q}_{\bp}(X)=0$ for $p\ge 0$ and $q\ge 1$. Since $\bp^{p,-1}=0$, (\ref{sduality}) holds for all $p$ and $q$.
In particular,  $H^{p,q}_{\bp,c}(X)=0$ for $p\ge 0$ and  $0\le q\le n-1$ and 
$H^{p,n}_{\bp,c}(X)\simeq H^{n-p,0}_{\bp}(X)'$.


\begin{theorem}\label{pbalex}
Suppose $X$ is Stein. Let $q$ be an integer with $q\ge 2$. Suppose that in the sequence
\[
\scE^{(p,q-2)}(X\ssm K)\overset{\bp^{p,q-2}}\lra \scE^{(p,q-1)}(X\ssm K)\overset{\bp^{p,q-1}}\lra \scE^{(p,q)}(X\ssm K),
\]
$\op{Im}\bp^{p,q-2}$ and $\op{Im}\bp^{p,q-1}$ are closed.
Then the groups $H^{p,q}_{\bar\vt}(X,X\ssm K)$ and $H^{n-p,n-q}_{\bp}[K]$ admit natural \str s of FS and DFS spaces, \r, and
\[
\bar A:H^{p,q}_{\bar\vt}(X,X\ssm K)\overset\sim\lra H^{n-p,n-q}_{\bp}[K]'.
\]
\end{theorem}

\begin{proof}
By assumption, we have the Serre duality for $X$ and, for $q\ge 2$, 
\[
\delta:H^{p,q-1}_{\bp}(X \ssm K)\overset\sim\lra H^{p,q}_{\bar\vt}(X,X\ssm K)\quad\text{and}\quad
\gamma^{*}:H^{n-p,n-q}_{\bp}[K]\overset\sim\lra H^{n-p,n-q+1}_{\bp,c}(X\ssm K).
\]
Also by assumption, $H^{p,q-1}_{\bp}(X \ssm K)$ is FS and 
$H^{n-p,n-q+1}_{\bp,c}(X\ssm K)$ is DFS. Thus if we endow $H^{p,q}_{\bar\vt}(X,X\ssm K)$ and $H^{n-p,n-q}_{\bp}[K]$
with FS and DFS \str s so that $\delta$ and $\gamma^{*}$ become \iso s, we have the duality.
\end{proof}



\lsection{Examples, applications and related topics}\label{secapp}

\subsection*{I. A canonical Dolbeault-\v Cech correspondence}

We consider the covering $\W'=\{W_i\}_{i=1}^n$ of $\C^n\ssm \{0\}$ given by
$W_i=\{z_i\ne 0\}$. 
Here we put ``$\,'\,$''  as we later consider the covering $\W=\{W_{i}\}_{i=1}^{n+1}$ of $\C^{n}$
with  $W_{n+1}=\C^{n}$ (cf. Remark~\ref{remdelta}.\,2 below). In the sequel we denote $\C^n\ssm \{0\}$ by $\C^n\ssm 0$.
We set
\[
\vPhi(z)=dz_{1}\wedge\cdots\wedge dz_{n}\quad\text{and}\quad\vPhi_{i}(z)=(-1)^{i-1}z_{i}\,dz_{1}\wedge\cdots\wedge \widehat{dz_{i}}\wedge\cdots\wedge dz_{n}.
\]
Then on the one hand we have the Bochner-Martinelli form
\[
\beta_n=C_{n}\frac{\sum_{i=1}^{n}\overline{\vPhi_{i}(z)}\wedge\vPhi(z)}{{\Vert z\Vert}^{2n}},
\qquad C_{n}=(-1)^{\frac {n(n-1)} 2}\frac{(n-1)!}{(2\pi\sqrt{-1})^{n}},
\]
which is a $\bp$-closed $(n,n-1)$-form on $\C^n\ssm 0$. On the other hand we have
the Cauchy form in $n$-variables
\[
\kappa_n=\Big(\frac{1}{2\pi\sqrt{-1}}\Big)^{n}\frac{\vPhi(z)}{z_1\cdots z_n},
\]
which may be thought of as a cocycle $c$ in $C^{n-1}(\W';\scO^{(n)})$ given by $c_{1\dots n}=\kappa_{n}$. 

\begin{theorem}\label{BMCauchy}  Under the \iso
\[
H^{n,n-1}_{\bp}(\C^n\ssm 0)\simeq H^{n-1}(\W';\scO^{(n)})
\]
of Corollary \ref{natisos}, the class of $\beta_n$ corresponds to the class of $(-1)^{\frac{n(n-1)} 2}\kappa_n$.
\end{theorem}

\begin{proof} If $n=1$, the cohomologies are the same and $\beta_1=\kappa_1$. Thus we assume $n\ge 2$. 
We may think of $\b_{n}$ as being in $C^{0}(\W';\scE^{(n,n-1)})\subset \scE^{(n,n-1)}(\W')$ and $\kappa_{n}$ in $C^{n-1}(\W';\scE^{(n,0)})\subset \scE^{(n,n-1)}(\W')$.
We construct a cochain $\chi$ 
in $\scE^{(n,n-2)}(\W')=\bigoplus_{p=0}^{n-2} C^p(\W';\scE^{(n,q)})$, $q=n-p-2$,  so that 
\[
\beta_n-(-1)^{\frac {n(n-1)} 2}\kappa_n=\bar\vartheta\chi\qquad\text{in}\ \ \scE^{(n,n-1)}(\W').
\]
Writing $\chi=\sum_{p=0}^{n-2}\chi^{p}$, $\chi^{p}\in C^p(\W';\scE^{(n,q)})$, this is  expressed as (cf. (\ref{dRCcorr}))

\begin{equationth}\label{CDcorresp}
\begin{cases}
  \beta_n=\bp\chi^{0},\\
  0=\check\delta\chi^{p-1}+(-1)^{p}\bp\chi^{p},\qquad 1\le p\le n-2,\\
 -(-1)^{\frac {n(n-1)} 2}\kappa_n =\check\delta\chi^{n-2}.
\end{cases}
\end{equationth}
Note that the  condition in the middle is vacuous if $n=2$.

 Let $0\le p\le n-2$ so that $0\le q\le n-2$. For a $(p+1)$-tuple of integers $I=(i_0,\dots,i_p)$ with $1\le i_0<\cdots <i_p\le n$, let $I^{*}=(j_0,\dots,j_q)$ denote the complement
of $\{i_0,\dots,i_p\}$ in $\{1,\dots,n\}$ with $1\le j_0<\dots <j_q\le n$.
Setting
$z_I=z_{i_0}\cdots z_{i_p}$, $|I^{*}|=j_0+\cdots+j_q$ and
\[
\bar\varPhi_{I^{*}}(z)=\sum_{\mu=0}^q(-1)^\mu\bar z_{j_\mu}d\bar z_{j_0}\wedge\cdots\wedge\widehat{d\bar z_{j_\mu}}\wedge\cdots\wedge d\bar z_{j_q},
\]
we define a cochain $\chi^{p}$ by
\[
\chi^{p}_I
=(-1)^{\varepsilon_{I}}\frac {q!\,C_{n}} {(n-1)!}\,\frac{\bar\varPhi_{I^{*}}(z)
\wedge\varPhi(z)}{z_I\,\Vert z\Vert^{2(q+1)}},\qquad \varepsilon_{I}=|I^{*}|+\frac{q(n+p-1)}2.
\]
and prove that  it satisfies  (\ref{CDcorresp}). 

If we set $d\bar z_{I^{*}}=d\bar z_{j_{0}}\wedge\cdots\wedge d\bar z_{j_{q}}$, we have
\[
\bar\partial\, \bar\varPhi_{I^{*}}(z)=(q+1)d\bar z_{I^{*}}.
\]
We also have 
\[
\bar\partial\,\frac 1 {\Vert z \Vert^{2(q+1)}}=-\frac {q+1} {\Vert z \Vert^{2(q+2)}}\sum_{i=1}^{n}z_{i}d\bar z_{i}.
\]
Using these, we compute
\[
\bar\partial\chi^{p}_{I}=(-1)^{\varepsilon_{I}}C_{n}\frac{(q+1)!}{(n-1)!}\,\frac{(\Vert z\Vert^{2}d\bar z_{I^{*}}-\sum_{i=1}^{n}z_{i}d\bar z_{i}\wedge\bar\varPhi_{I^{*}}(z))
\wedge\varPhi(z)}{z_I\,\Vert z\Vert^{2(q+2)}}.
\]
To see the numerator, note that
\[
\sum_{i=1}^{n}z_{i}d\bar z_{i}\wedge\bar\varPhi_{I^{*}}(z)
=\sum_{\nu=0}^{p}z_{i_{\nu}}d\bar z_{i_{\nu}}\wedge\bar\varPhi_{I^{*}}(z)+\sum_{\mu=0}^{q}|z_{j_{\mu}}|^2d\bar z_{I^{*}}
\]
so that
\[
\Vert z\Vert^{2}d\bar z_{I^{*}}-\sum_{i=1}^{n}z_{i}d\bar z_{i}\wedge\bar\varPhi_{I^{*}}(z)
=\sum_{\nu=0}^{p}z_{i_{\nu}}(\bar z_{i_{\nu}}d\bar z_{I^{*}}-d\bar z_{i_{\nu}}\wedge\bar\varPhi_{I^{*}}(z)).
\]
Thus, setting $I_{\nu}=(i_{0},\dots,\widehat{i_{\nu}},\dots,i_{p})$, we have
\begin{equationth}\label{dbartheta2}
\bar\partial\chi^{p}_{I}=(-1)^{\varepsilon_{I}}C_{n}\frac{(q+1)!}{(n-1)!}\,\sum_{\nu=0}^{p}\frac{(\bar z_{i_{\nu}}d\bar z_{I^{*}}-d\bar z_{i_{\nu}}\wedge\bar\varPhi_{I^{*}}(z))
\wedge\varPhi(z)}{z_{I_{\nu}}\,\Vert z\Vert^{2(q+2)}}.
\end{equationth}

Now we verify the three identities in (\ref{CDcorresp}) successively.
\vv

\noindent
(I) First identity. If $p=0$, then $q=n-2$.  In this case $I$ is of the form $(r)$,  $r=1,\dots,n$, and  $(r)^{*}=(1,\dots,\widehat{r},\dots,n)$. 
Denoting $(r)$ by $r$, from (\ref{dbartheta2}), we have
\[
\bar\partial\chi^{0}_{r}=(-1)^{\varepsilon_{r}}C_{n}\,\frac{(\bar z_{r}d\bar z_{(r)^{*}}-d\bar z_{r}\wedge\bar\varPhi_{(r)^{*}}(z))
\wedge\varPhi(z)}{\Vert z\Vert^{2n}}.
\]
On the other hand, we compute
\[
\sum_{i=1}^{n}\overline{\varPhi_{i}(z)}=\overline{\varPhi_{r}(z)}+\sum_{i\ne r}\overline{\varPhi_{i}(z)}
=(-1)^{r-1}(\bar z_{r}d\bar z_{(r)^{*}}-d\bar z_{r}\wedge\bar\varPhi_{(r)^{*}}(z)).
\]
Noting that $\varepsilon_{r}-r+1=n(n-1)-2(r-1)$, which is always even,
we have the first identity.
\vv

\noindent
(II) The second identity. This applies for $n\ge 3$. Suppose $1\le p\le n-2$ so that $0\le q\le n-3$. By definition
\[
(\check\delta\chi^{p-1})_{I}=\sum_{\nu=0}^{p}(-1)^{\nu}\chi^{p-1}_{I_{\nu}}.
\]
We have
\[
\chi^{p-1}_{I_{\nu}}
=(-1)^{\varepsilon_{I_{\nu}}}\frac {(q+1)!\,C_{n}} {(n-1)!}\,\frac{\bar\varPhi_{I_{\nu}^{*}}(z)
\wedge\varPhi(z)}{z_{I_{\nu}}\,\Vert z\Vert^{2(q+2)}}.
\]
To compute $\bar\varPhi_{I_{\nu}^{*}}(z)$, let $r_{\nu}$ denote the integer with $-1\le r_{\nu}\le q$ \st\
$j_{r_{\nu}}<i_{\nu}<j_{r_{\nu}+1}$, where we set $j_{-1}=0$ and $j_{q+1}=n+1$. Then we have
\[
\bar\varPhi_{I_{\nu}^{*}}(z)=(-1)^{r_{\nu}+1}(\bar z_{i_{\nu}}d\bar z_{I^{*}}-d\bar z_{i_{\nu}}\wedge\bar\varPhi_{I^{*}}(z)).
\]
Thus, comparing with (\ref{dbartheta2}), it suffices to show that the parity of $\nu+\varepsilon_{I_{\nu}}+r_{\nu}+1$ is different from that of
$p+\varepsilon_{I}$.
We have
\[
\varepsilon_{I_{\nu}}=|I_{\nu}^{*}|+\frac {(q+1)(n+p-2)} 2=|I^{*}|+i_{\nu}+\frac {(q+1)(n+p-2)} 2.
\]
Therefore
\[
\nu+\varepsilon_{I_{\nu}}+r_{\nu}+1-(p+\varepsilon_{I})=\nu+i_{\nu}+r_{\nu}+1.
\]
We show, by induction on $\nu$, that $\nu+i_{\nu}+r_{\nu}$ is even. Suppose $\nu=0$. If $i_{0}<j_{0}$, then $i_{0}=1$
and $r_{0}=-1$ so that it is even. If $i_{0}>j_{0}$, then $i_{0}=r_{0}+2$ and it is even. Suppose it is even
for $\nu$. If $i_{\nu+1}<j_{r_{\nu}+1}$, then $i_{\nu+1}=i_{\nu}+1$ and $r_{\nu+1}=r_{\nu}$ so that
it is even for $\nu+1$. If $i_{\nu+1}>j_{r_{\nu}+1}$, then $i_{\nu+1}=i_{\nu}+r_{\nu+1}-r_{\nu}+1$
and it is even.

\vv

\noindent
(III) The third identity.
If $p=n-2$, then $q=0$ and, for $r=1,\dots,n$, we set $I^{(r)}=(1,\dots,\widehat{r},\dots,n)$. Then
${I^{(r)}}^{*}=(r)$ and we have
\[
\chi^{n-2}_{I^{(r)}}=(-1)^{r+\frac{n(n-1)}2}\Big(\frac 1 {2\pi\sqrt{-1}}\Big)^{n}\,\frac{\bar z_{r}
\varPhi(z)}{z_{I^{(r)}}\,\Vert z\Vert^{2}}=(-1)^{r+\frac{n(n-1)}2}\Big(\frac 1 {2\pi\sqrt{-1}}\Big)^{n}\,\frac{ |z_{r}|^{2}
\varPhi(z)}{\Vert z\Vert^{2}\, z_{1}\cdots z_{n}}.
\]
By definition
\[
\check\delta\chi^{n-2}_{1\dots n}=\sum_{r=1}^{n}(-1)^{r-1}\chi^{n-2}_{I^{(r)}}=-(-1)^{\frac {n(n-1)} 2}\kappa_{n}.
\]
\end{proof}

\begin{remark}\label{rembm0} {\bf 1.} If we set
\[
\beta_n^{0}=C_{n}\frac{\sum_{i=1}^{n}\overline{\vPhi_{i}(z)}}{{\Vert z\Vert}^{2}},\qquad
\kappa_n^{0}=\Big(\frac 1 {2\pi\sqrt{-1}}\Big)^{n}\,\frac{1}{z_1\cdots z_n},
\]
under the \iso
\[
H^{0,n-1}_{\bp}(\C^n\ssm 0)\simeq H^{n-1}(\W';\scO),
\]
the class of $\beta_n^{0}$ corresponds to the class of $(-1)^{\frac{n(n-1)} 2}\kappa_n^{0}$.
\smallskip

\noindent
{\bf 2.} Let $V$ be a Stein \nbd\ of $0$ in $\C^{n}$ and $\V'=\{V_i\}_{i=1}^n$ the covering of $V\ssm 0$ given by
$V_i=V\cap W_{i}$. Then we have a canonical \iso\ 
$H^{n,n-1}_{\bp}(V\ssm 0)\simeq H^{n-1}(\V';\scO^{(n)})$, under which the  the class of 
$\b_{n}$ (restricted to $V\ssm 0$) corresponds to the  class of  $(-1)^{\frac{n(n-1)} 2}\kappa_n$ (restricted to $\V'$).
Suppose the class of $\t$ corresponds to the class of $\gamma$ under the above \iso.
If $h$ is a \h\ \fcn\ on $V$, since $\bp h=0$, we see that the class of $h\t$ corresponds to the class of $h\gamma$ (cf. \eqref{dRCcorr}, \eqref{CDcorresp}). 
\end{remark}


In the sequel we endow $\C^{n}=\{(z_{1},\dots,z_{n})\}$, $z_{i}=x_{i}+\sqrt{-1}y_{i}$, with the usual 
orientation, i.e., the one where
 $(x_{1},y_{1},\dots,x_{n},y_{n})$ is a positive coordinate system. In the above situation set
\[
R_{1}=\{\,z\in\C^{n}\mid \Vert z\Vert^{2}\le n\ve^{2}\,\}.
\]
The  boundary $\partial R_{1}$ is a usually oriented $(2n-1)$-sphere $S^{2n-1}$.
We also set
\[
\vG=\{\,z\in\C^{n}\mid |z_{i}|=\ve,\ i=1,\dots,n\,\},
\]
which is an $n$-cycle oriented so that $\op{arg}z_{1}\wedge\cdots \wedge\op{arg}z_{n}$ is positive.

\begin{theorem}\label{thcdcorrint} Let $\t$ be a $\bp$-closed $(n,n-1)$-form on $\C^n\ssm 0$ and $\gamma$ a cocycle in $C^{n-1}(\W';\scO^{(n)})$. If the class of $\t$ corresponds to the class of $\gamma$ by the canonical \iso
\[
H^{n,n-1}_{\bp}(\C^n\ssm 0)\simeq H^{n-1}(\W';\scO^{(n)}),
\]
then
\[
\int_{S^{2n-1}}\t=(-1)^{\frac {n(n-1)} 2}\int_{\vG}\gamma.
\]
\end{theorem}

\begin{proof} Recall that we have canonical \iso s
\[
H^{n,n-1}_{\bp}(\C^n\ssm 0)\overset\sim\lra H^{n,n-1}_{\bar\vt}(\W')\overset\sim\longleftarrow  H^{n-1}(\W';\scO^{(n)}).
\]
The assumption implies that there exists a cochain $\chi$ in $\scE^{(n,n-2)}(\W')$ \st\ 
$\t-\gamma=\bar\vt\chi$.
Consider the commutative diagram
\[
\SelectTips{cm}{}
\xymatrix
@C=.7cm
@R=.6cm
{\scE^{(n,n-2)}(\W')\ar[r]^-{} \ar[d]^{\bar\vt}&\scE^{(2n-2)}(\W')\ar[r]^-{}\ar[d]^-{D} &\scE^{(2n-2)}(\W'\cap S^{2n-1})\ar[d]^-{D}\\
\scE^{(n,n-1)}(\W')\ar[r]^-{} & \scE^{(2n-1)}(\W')\ar[r]^-{}&\scE^{(2n-1)}(\W'\cap S^{2n-1})\ar[r]^-{\int_{S^{2n-1}}}&\C,}
\]
where $\W'\cap S^{2n-1}$ denotes the covering of $ S^{2n-1}$ consisting of the $W_{i}\cap S^{2n-1}$'s.
For each $i=1,\dots,n$, we set
\[
Q_{i}=\{\,z\in  S^{2n-1}\mid |z_{i}|\ge |z_{j}|\ \ \text{for all}\ j\ne i\,\}.
\]
Then $\{Q_{i}\}$ is a honeycomb system adapted to $\W'\cap S^{2n-1}$ and, by the Stokes formula
for \v{C}ech-de~Rham cochains,
\[
0=\int_{S^{2n-1}}D\chi=\sum_{p=0}^{n-1}\sum_{i_{0}<\cdots<i_{p}}\int_{Q_{i_{0}\dots i_{p}}}(D\chi)_{i_{0}\dots i_{p}}=\sum_{i=1}^{n}\int_{Q_{i}}\t-\int_{Q_{1\dots n}}\gamma.
\]
Noting that $Q_{1\dots n}=(-1)^{\frac{n(n-1)}2}\vG$, we have the theorem.
\end{proof}

Note that the above  is consistent with Theorem \ref{BMCauchy}\,:
\[
\int_{S^{2n-1}}\b_{n}=1=\int_{\vG}\kappa_{n}.
\]

\subsection*{II. Local duality}

\paragraph{A theorem of Martineau\,:} The following theorem of A. Martineau \cite{M} (see also \cite{H1}, \cite{K}) may
naturally be  interpreted in our framework as one of the  cases where the $\bp$-Alexander morphism is an \iso\ 
with topological duals so that the duality
pairing is given by the cup product followed by integration as described in Section \ref{secld}.

In the below we assume that $\C^{n}$ is oriented, but the orientation may not be the usual one.

\begin{theorem}\label{thMH} Let $K$ be a compact set in $\C^{n}$ \st\ $H^{p,q}_{\bp}[K]=0$ for $q\ge 1$.
Then for any open set $V\supset K$, $H^{p,q}_{\bar\vt}(V,V\ssm K)$ and $H^{n-p,n-q}_{\bp}[K]$ admits 
natural \str s of FS and DFS spaces, \r, and we have\,{\rm :}
\[
\bar A:H^{p,q}_{\bar\vt}(V,V\ssm K)\overset\sim\lra H^{n-p,n-q}_{\bp}[K]'=\begin{cases} 0\qquad & q\ne n\\
                                                                       \scO^{(n-p)}[K]'\qquad & q=n.
                                                                       \end{cases}
\]
\end{theorem}

The  theorem is originally stated for $p=0$ in terms of local cohomology.
 This is proved by applying Theorem \ref{pbalex}. First, by excision 
we may assume that $V$ is Stein. Thus the essential point is to prove that the hypothesis of Theorem \ref{pbalex} holds, which is done using
a theorem of Malgrange (cf. \cite{K}). Incidentally, the hypothesis $H^{p,q}_{\bp}[K]=0$, for $q\ge 1$, is satisfied
if $K$ is a subset of $\R^{n}$ by the following theorem (cf. \cite{Gr})\,:

\begin{theorem}[Grauert] \label{thGrauert} Any subset  of $\R^{n}$ admits a fundamental system of  \nbd s consisting of Stein open sets in $\C^{n}$.\label{thGrauert}
\end{theorem}

In our framework, the duality is described as follows (cf. Section \ref{secld}). Let $V_{0}=V\ssm K$ and $V_{1}$ a \nbd\ of $K$ in $V$ and consider the coverings $\V_{K}=\{V_{0},V_{1}\}$ and $\V'_{K}=\{V_{0}\}$ of $V$ and $V_{0}$.
The duality pairing is give, for a  cocycle $(\xi_{1},\xi_{01})$ in $\scE^{(p,n)}(\V_{K},\V'_{K})$ and a \h\ $(n-p)$-form $\eta$
near $K$, by
\begin{equationth}\label{locpair}
\int_{R_{1}}\xi_{1}\wedge\eta+\int_{R_{01}}\xi_{01}\wedge\eta,
\end{equationth}
where $R_{1}$ is a compact real $2n$-dimensional \mfd\ with $C^{\infty}$ boundary in $V_{1}$ containing $K$ in its interior and
$R_{01}=-\partial R_{1}$. 
We may always choose a cocycle  so that $\xi_{1}=0$ if $V$ is Stein.

\paragraph{Local residue pairing\,:} 
Now we consider Theorem \ref{thMH} in the case $K=\{0\}$ and $(p,q)=(n,n)$. We also let $V=\C^{n}$.
In this paragraph we consider the usual orientation on $\C^{n}$.
We have the exact sequence
\[
\cdots\lra H^{n,n-1}_{\bp}(\C^{n}\ssm 0)\overset\delta\lra H^{n,n}_{\bar\vt}(\C^{n},\C^{n}\ssm 0)\lra 0.
\]
Thus every element in $H^{n,n}_{\bar\vt}(\C^{n},\C^{n}\ssm 0)$ is represented by a cocycle of the form $(0,-\t)$.
Since $\scO[K]=\scO_{\C^{n},0}=\scO_{n}$ in this case, the duality in Theorem \ref{thMH} is induced by the pairing
\[
H^{n,n}_{\bar\vt}(\C^{n},\C^{n}\ssm 0)\times \scO_{n}\overset{\int}\lra\C
\]
given by
\[
((0,-\t),h)\mapsto -\int_{R_{01}}h\t=\int_{S^{2n-1}}h\t.
\]
In the above, $h$ is a \h\ \fcn\ in a \nbd\ $V$ of $0$ in $\C^{n}$. We may take as $R_{1}$ a $2n$-ball around $0$ in $V$ so that $R_{01}=-\partial R_{1}=-S^{2n-1}$ with $S^{2n-1}$  a usually oriented $(2n-1)$-sphere.
Thus if $\t$ corresponds to $\gamma$, the above integral is equal to 
\[
(-1)^{\frac {n(n-1)} 2}\int_{\vG}h\gamma
\]
(cf. Remark \ref{rembm0}.\,2 and Theorem \ref{thcdcorrint}). In particular, if $\t=\b_{n}$ the pairing is given by
\begin{equationth}\label{lrp}
\int_{S^{2n-1}}h\b_{n}=\int_{\vG}h\kappa_{n}=\Big(\frac{1}{2\pi\sqrt{-1}}\Big)^n\int_{\vG} \frac{hdz_{1}\wedge
\cdots\wedge dz_{n}}{z_1\cdots z_n}
=h(0).
\end{equationth}

Likewise in the case $(p,q)=(0,n)$, the duality in Theorem \ref{thMH} is induced by the pairing
\[
H^{0,n}_{\bar\vt}(\C^{n},\C^{n}\ssm 0)\times \scO^{(n)}_{\C^{n},0}\overset{\int}\lra\C
\]
given by
\[
((0,-\t),\eta)\mapsto -\int_{R_{01}}\t\wedge\eta=\int_{S^{2n-1}}\t\wedge\eta.
\]

\subsection*{III. Sato hyperfunctions}

Sato hyperfunctions are defined in terms of relative cohomology with coefficients in the sheaf of \h\ \fcn s and the
theory is developed in the language of derived functors (cf. \cite{Sa}, \cite{Skk}). The use of relative Dolbeault cohomology via the relative Dolbeault theorem (Theorems \ref{thdrel}, \ref{thdrel2}) provides us with another way of treating the 
theory. This approach gives simple and explicit expressions of hyperfunctions and some fundamental operations on them and leads
to 
a number of new results. These are discussed in detail in \cite{HIS}. Here we pick up some of the essentials of the contents therein. In general
the theory of hyper\fcn s may be developed on an arbitrary real analytic \mfd\ and it  involves various orientation sheaves.
However for simplicity, here we consider hyper\fcn s  on open sets in $\R^{n}$  fixing various orientations.

\paragraph{Hyper\fcn s and hyperforms\,:} We consider the standard inclusion $\R^{n}\subset\C^{n}$, i.e.,
if  $(z_{1},\dots,z_{n})$, $z_{i}=x_{i}+\sqrt{-1}y_{i}$, is a coordinate system on $\C^{n}$, then $\R^{n}$ is given by $y_{i}=0$, $i=1,\dots,n$.  We orient $\R^{n}$ and $\C^{n}$
so that $(x_{1},\dots,x_{n})$ and $(y_{1},\dots,y_{n},x_{1},\dots,x_{n})$ are positive coordinate
systems. Thus $(y_{1},\dots,y_{n})$ is a positive coordinate system in the normal direction.
Note that the difference between this orientation of $\C^{n}$ and the usual one, where $(x_{1},y_{1},\dots,x_{n},y_{n})$ is positive, is the sign of $(-1)^{\frac{n(n+1)} 2}$.

With these, for an open set $U$ in $\R^{n}$, the space of hyperfunctions on $U$ is given by
\[
\scB(U)=H^{n}_{U}(V;\scO),
\]
where $V$ is an open set in $\C^{n}$ containing $U$ as a closed set and $\scO$ the sheaf of \h\ \fcn s on $\C^{n}$. We call such a $V$ a complex \nbd\ of $U$. Note that, by excision, the definition does not depend on the choice of the complex \nbd\ $V$. By the relative Dolbeault theorem (cf. Theorems \ref{thdrel} and \ref{thdrel2}), there is a 
canonical \iso\,:
\[
\scB(U)\simeq H^{0,n}_{\bar\vt}(V,V\ssm U).
\]
More generally we introduce the following\,:
\begin{definition} The space of {\em $p$-hyperforms} on $U$ is defined by
\[
\scB^{(p)}(U)=H^{p,n}_{\bar\vt}(V,V\ssm U).
\]
\end{definition}

Note that the definition does not depend on the choice of $V$ by excision (cf. Proposition~\ref{excision}). Denoting
by $\scO^{(p)}$ the sheaf of \h\ $p$-forms on $\C^{n}$, we have a canonical \iso\ (cf. Theorems
\ref{thdrel} and \ref{thdrel2})\,:
\[
H^{p,n}_{\bp}(V,V\ssm U)\simeq H^{n}_{U}(V;\scO^{(p)})
\]
so that $\scB^{(0)}(U)$ is canonically isomorphic with $\scB(U)$. Hyperforms are essentially equal to what have conventionally  been referred
to as differential forms with coefficients in hyper\fcn s.


\begin{remark}\label{rempc} In the above we implicitly used the fact that $\R^{n}$ is ``purely $n$-codimensional'' in $\C^{n}$ \wrt\
$\scO^{(p)}_{\C^{n}}$ and $\Z_{\C^{n}}$ (cf. \cite{KKK}). For the latter, this can be seen from the Thom \iso\ (cf. Subsection \ref{ssrdr}).
\end{remark}


\paragraph{Expression of hyperforms\,:} Let $U$ and $V$ be as above.
Letting $V_{0}=V\ssm U$ and $V_{1}$ a \nbd\ of $U$ in $V$, we  consider the open coverings $\V=\{V_{0},V_{1}\}$ 
and $\V'=\{V_{0}\}$ of $V$ and $V\ssm U$. 
Then $\scB^{(p)}(U)=H^{p,n}_{\bar\vt}(V,V\ssm U)=H^{p,n}_{\bar\vt}(\V,\V')$ and a $p$-hyperfom is represented by a pair $(\xi_{1},\xi_{01})$ with $\xi_{1}$ a 
$(p,n)$-form on $V_{1}$, which is automatically $\bp$-closed, and  $\xi_{01}$ a $(p,n-1)$-form on $V_{01}$ \st\ $\xi_{1}=\bar\partial\xi_{01}$ on $V_{01}$. 
We have the exact sequence (cf. \eqref{lexactrelD}, \eqref{lexactrelD2})\,:
\[
H_{\bp}^{p,n-1}(V)\lra H^{p,n-1}_{\bar\vt}(V\ssm U)\overset{\delta}\lra \scB^{(p)}(U)\overset{j^{*}}\lra H^{p,n}_{\bp}(V).
\]

By Theorem \ref{thGrauert}, we may take as $V$ a Stein open set and, if we do this,
we have $H^{p,n}_{\bp}(V)\simeq H^{n}(V,\scO^{(p)})=0$. Thus $\delta$ is
surjective and every $p$-hyperform is represented by a cocycle of the form $(0,-\t)$ with $\t$ a $\bar\partial$-closed
$(p,n-1)$-form on $V\ssm U$.

In the case $n>1$, $H^{p,n-1}_{\bp}(V)\simeq H^{n-1}(V,\scO^{(p)})=0$ and $\delta$ is an \iso. In the case $n=1$,
we have the exact sequence
\[
H_{\bp}^{p,0}(V)\lra H^{p,0}_{\bar\vt}(V\ssm U)\overset{\delta}\lra \scB^{(p)}(U)\lra 0,
\]
where $H^{p,0}_{\bar\vt}(V\ssm U)\simeq H^{0}(V\ssm U,\scO^{(p)})$ and 
$H^{p,0}_{\bp}(V)\simeq H^{0}(V,\scO^{(p)})$. Thus, for $p=0$,  we recover the original expression of hyperfunctions by Sato
in one dimensional case.

\begin{remark} Although a hyperform may be represented by a single differential form in most of the cases, it is important to keep in mind
that it is represented by a pair $(\xi_{1},\xi_{01})$ in general.
\end{remark}



Now we describe some of the operations on hyperforms.

\paragraph{Multiplication by real analytic \fcn s\,:} Let $\scA(U)$ denote the space of real analytic \fcn s on 
$U$. We define
the multiplication
\[
\scA(U)\times \scB^{(p)}(U)\lra \scB^{(p)}(U)
\]
by assigning to $(f,[\xi])$ the class of $(\tilde f\xi_{1},\tilde f\xi_{01})$ with $\tilde f$ a \h\ extension
of $f$. 

\paragraph{Partial derivatives\,:} 
We  define the partial derivative
\[
\frac \partial {\partial x_{i}}:\scB(U)\lra \scB(U)
\]
as follows. Let $(\xi_{1},\xi_{01})$ represent a hyper\fcn\ on $U$. We write $\xi_{1}=f\,d\bar z_{1}\wedge\cdots
\wedge d\bar z_{n}$
and 
$\xi_{01}=\sum_{j=1}^{n}g_{j}\,d\bar z_{1}\wedge\cdots\wedge\widehat{d\bar z_{j}}\wedge\cdots
\wedge d\bar z_{n}$.
Then $\frac \partial {\partial x_{i}}[\xi]$
is represented by the cocycle
\[
\bigl(\frac{\partial f}{\partial z_{i}}\,d\bar z_{1}\wedge\cdots
\wedge d\bar z_{n},\sum_{j=1}^{n}\frac{\partial g_{j}}{\partial z_{i}}\,d\bar z_{1}\wedge\cdots\wedge\widehat{d\bar z_{j}}\wedge\cdots
\wedge d\bar z_{n}\bigr).
\]

Thus for a differential operator $P(x,D)$, $P(x,D):\scB(U)\ra \scB(U)$ is well-defined.

\paragraph{Restriction\,:} Let $U'$ be an open subset of $U$. Take a complex \nbd\ $V'$ of $U'$ and a \nbd\ 
$V_{1}'$ of
$U'$ in $V'$ so that $V'\subset V$ and $V'_{1}\subset V_{1}$. Then the restriction  $\scB^{(p)}(U)\ra \scB^{(p)}(U')$
is defined by assigning to the class of $(\xi_{1},\xi_{01})$ the class of 
$(\xi_{1}|_{V'_{1}},\xi_{01}|_{V'_{01}})$, $V'_{01}=V'_{1}\ssm U'$.

\paragraph{Differential\,:} We  define the differential
\[
d:\scB^{(p)}(U)\lra\scB^{(p+1)}(U).
\]
by assigning to the class of $(\xi_{1},\xi_{01})$ the class of $(-1)^{n}(\partial\xi_{1},-\partial\xi_{01})$.  
From Theorem \ref{thexactpcd}, we have the exact sequence (de Rham complex for hyperforms, cf. Remark \ref{rempc})\,:
\begin{equationth}\label{hypdR}
0\lra\C(U)\lra \scB(U)\overset{d}\lra\scB^{(1)}(U)\overset{d}\lra\cdots\overset{d}\lra\scB^{(n)}(U)\lra 0.
\end{equationth}
We come back to the first part  below.
\vv

\paragraph{Integration\,:} 
Let  $K$ be a compact set in $U$. 
We define the space $\scB^{(p)}_{K}(U)$ of $p$-hyperforms on $U$ with support in $K$ as the kernel
of the restriction $\scB^{(p)}(U)\ra \scB^{(p)}(U\ssm K)$.
Then we have\,:

\begin{proposition} For any complex \nbd\ $V$ of $U$, there is a canonical \iso
\[
\scB^{(p)}_{K}(U)\simeq H^{p,n}_{\bar\vt}(V,V\ssm K).
\]
\end{proposition}
\begin{proof} Applying Proposition \ref{proptriplerd} for the triple $(V,V\ssm K,V\ssm U)$, we have the exact sequence
\[
 H^{p,n-1}_{\bar\vt}(V\ssm K,V\ssm U)\overset{\delta}\lra H^{p,n}_{\bar\vt}(V,V\ssm K)\overset{j^{*}}\lra H^{p,n}_{\bar\vt}(V,V\ssm U)\overset{i^{*}}\lra H^{p,n}_{\bar\vt}(V\ssm K,V\ssm U).
\]
By definition, $H^{p,n}_{\bar\vt}(V,V\ssm U)=\scB^{(p)}(U)$. Since $V\ssm K$ is a complex \nbd\ of $U\ssm K$
and $(V\ssm K)\ssm (U\ssm K)=V\ssm U$, $H^{p,n}_{\bar\vt}(V\ssm K,V\ssm U)=\scB^{(p)}(U\ssm K)$.
On the other hand, $H^{p,n-1}_{\bar\vt}(V\ssm K,V\ssm U)=0$ (cf. Remark \ref{rempc}).
\end{proof}

By the above proposition, we may define the integration on $\scB^{(n)}_{K}(U)$ by directly applying
\eqref{intDrel}, which we recall for the sake of completeness.
Let $V$ be a complex \nbd\ of $U$ and consider the coverings $\V_{K}=\{V_{0},V_{1}\}$ and $\V'_{K}=\{V_{0}\}$,
with $V_{0}=V\ssm K$ and $V_{1}$ a \nbd\ of $K$ in $V$. Then we have a canonical identification 
$\scB^{(p)}_{K}(U)=H^{p,n}_{\bar\vt}(\V_{K},\V'_{K})$. Let $R_{1}$ be a compact real $2n$-dimensional \mfd\ with $C^{\infty}$ boundary  in $V_{1}$ containing $K$ in its interior and
set $R_{01}=-\partial R_{1}$.
Then the integration 
\[
\int_{U}:\scB^{(n)}_{K}(U)\lra\C
\]
is given as follows.
Noting that $u\in \scB^{(n)}_{K}(U)=H_{\bar\vt}(\V_{K},\V'_{K})$
is represented by
\[
\xi = (\xi_1,\xi_{01}) \in \scE^{(n,n)}(\V_{K},\V'_{K})=\scE^{(n,n)}(V_{1})\oplus\scE^{(n,n-1)}(V_{01}),
\]
we have
\[
\int_{U}u=\int_{R_{1}}\xi_{1}+\int_{R_{01}}\xi_{01}.
\]
\paragraph{Duality\,:}
By Theorem \ref{thMH} we have
\begin{equationth}\label{dualloc}
\scB^{(p)}_{K}(U)=H^{p,n}_{\bp}(V,V\ssm K)\simeq \scO^{(n-p)}[K]'=\scA^{(n-p)}[K]',
\end{equationth}
where $\scA^{(n-p)}$ denotes the sheaf of germs of real analytic $(n-p)$-forms on $\R^{n}$ and
\[
\scA^{(n-p)}[K]=\lim_{\lra}\scA^{(n-p)}(U_{1}),
\]
the direct limit over the set of \nbd s $U_{1}$ of $K$ in $U$. Recall that the pairing is induced by \eqref{locpair}.




\paragraph{$\delta$-\fcn\ and $\delta$-form\,:} We consider the case $K=\{0\}\subset\R^{n}$.

\begin{definition}\label{defdeltafcn} The {\em $\delta$-\fcn} is the element in 
$\scB_{\{0\}}(\R^{n})=H^{0,n}_{\bar\vt}(\C^{n},\C^{n}\ssm\{0\})$
which is represented by
\[
(0,-(-1)^{\frac {n(n+1)}2}\b_{n}^{0}),
\]
where $\b_{n}^{0}$ is as defined in Remark \ref{rembm0}.\,1.
\end{definition}

The \iso\
(\ref{dualloc})  reads in this case\,:
\[
\scB_{0}(\R^{n})\simeq (\scA^{(n)}_{0})',
\]
where $\scA^{(n)}_{0}$ denotes the stalk of $\scA^{(n)}$ at $0$.
For a  representative $\o=h(x)\varPhi(x)$ of an element in $\scA^{(n)}_{0}$, $h(z)\varPhi(z)$ is its complex representative. Let $R_{1}$ be a small $2n$-ball around $0$ in $\C^{n}$ so that $R_{01}=-\partial R_{1}
=-(-1)^{\frac {n(n+1)} 2} S^{2n-1}$ with $S^{2n-1}$  a usually oriented $(2n-1)$-sphere. 
Then the $\delta$-\fcn\ is the hyperfunction that assigns to a representative $\o=h(x) \varPhi(x)$ the value
(cf. \eqref{lrp})
\[
-(-1)^{\frac {n(n+1)} 2}\int_{R_{01}}h(z)\b_{n}=\int_{S^{2n-1}}h(z)\b_{n}=h(0).
\]

\begin{definition} The {\em $\delta$-form} is the element in 
$\scB^{(n)}_{\{0\}}(\R^{n})=H^{n,n}_{\bar\vt}(\C^{n},\C^{n}\ssm
\{0\})$
which is
represented by
\[
(0,-(-1)^{\frac {n(n+1)}2}\b_{n}).
\]
\end{definition}

Recall the \iso\
(\ref{dualloc}), which reads in this case\,:
\[
\scB_{0}^{(n)}(\R^{n})\simeq (\scA_{0})'.
\]
For a  representative $h(x)$ of an element in $\scA_{0}$, $h(z)$ is its complex representative. Let $R_{1}$ be as above.
Then the $\delta$-form is the hyperform that assigns to a representative $h(x)$ the value
\[
-(-1)^{\frac {n(n+1)} 2}\int_{R_{01}}h(z)\b_{n}=\int_{S^{2n-1}}h(z)\b_{n}=h(0).
\]



\begin{remark}\label{remdelta}{\bf 1.} If we orient $\C^{n}$ so that the usual coordinate system $(x_{1},y_{1},\dots,x_{n},y_{n})$ is positive, the delta \fcn\ $\delta(x)$ is represented by 
$(0,-\b_{n}^{0})$. Also, the delta form is represented by 
$(0,-\b_{n})$. Incidentally, it has the same expression as the Thom class of the trivial complex \vb\ of rank $n$
(cf. \cite[Ch.III, Remark 4.6]{Su2}).
\smallskip

\noindent
{\bf 2.} Set $W_{i}=\{z_{i}\ne 0\}$, $i=1,\dots,n$, and $W_{n+1}=\C^{n}$ and consider the coverings $\W=\{W_{i}\}_{i=1}^{n+1}$ and 
$\W'=\{W_{i}\}_{i=1}^{n}$ of $\C^{n}$  and $\C^{n}\ssm 0$. 
We have the natural \iso s
\[
\scB_{\{0\}}(\R^{n})\simeq H^{0,n-1}_{\bp}(\C^{n}\ssm 0)\simeq H^{n-1}(\W';\scO)\simeq H^{n}(\W,\W';\scO).
\]
As noted in Remark \ref{rembm0}.\,1, under the middle \iso\ above, the class of $\b_{n}^{0}$
corresponds to the class of 
$(-1)^{\frac {n(n-1)}2}\kappa_{n}^{0}$.
If we choose the usual orientation on $\C^{n}$, the class corresponding to $[\kappa_{n}^{0}]$ in $H^{n}(\W,\W';\scO)$ is the traditional representation of the $\delta$-\fcn\ (cf. \eqref{lrp}).
\end{remark}

\paragraph{1 as a hyperfunction\,:} We examine the map $\C(U)\ra\scB(U)$ in (\ref{hypdR}).
Let $\V$ and $\V'$ be as before. Then it is given by $\rho^{n}:H^{n}_{D}(\V,\V')\ra H^{0,n}_{\bar\vt}(\V,\V')$, which is induced by $(\o_{1},\o_{01})\mapsto (\o_{1}^{(0,n)},\o_{01}^{(0,n-1)})$ (cf. Corollary \ref{cordRD}). For simplicity
we assume that $U$ is connected.
Then we have the commutative diagram\,:
 \[
\SelectTips{cm}{}
\xymatrix@C=.6cm
@R=.4cm
{\C= H^{0}(U;\C)\ar[r]^-{\sim}_-{T}&H^{n}(V,V\ssm U;\C)\ar[r]^-{\iota}\ar@{-}[d]^-{\wr}& 
H^{n}(V,V\ssm U;\scO)=\scB(U)\ar@{-}[d]^-{\wr}\\
{} &H^{n}_{D}(\V,\V')\ar[r]^-{\rho^{n}}& H^{0,n}_{\bar\vt}(\V,\V'),}
\]
where $T$ denotes the Thom \iso,
which sends $1\in\C$ 
 to the Thom
class $\varPsi_{U}\in H^{n}(V,V\ssm U;\C)$ (cf. \eqref{thom}). If $\vP_{U}$ is represented by $(\psi_{1},\psi_{01})$ in $H^{n}_{D}(\V,\V')$, as a hyper\fcn, $1$ is represented by $\rho^{n}(\psi_{1},\psi_{01})=
(\psi_{1}^{(0,n)},\psi_{01}^{(0,n-1)})$. In particular, we may set $(\psi_{1},\psi_{01})=(0,-\psi_{n}(y))$,
where $\psi_{n}(y)$ is the angular form on $\R^{n}_{y}$ (cf. Proposition \ref{thtrivial}). Thus as a hyper\fcn, $1$ is represented by
$(0,-\psi_{n}^{(0,n-1)})$.
Noting that $y_i = 1/(2\sqrt{-1})(z_i - \bar{z}_i)$,
we compute
\[
\psi_{n}^{(0,n-1)}=(\sqrt{-1})^n C_n \dfrac{
	\sum_{i=1}^n 
	(-1)^i(z_i- \bar{z}_i) d\bar{z}_1 \wedge \dots \wedge 
	\widehat{d\bar{z}_i} \wedge \dots \wedge
	d\bar{z}_n}{\Vert z-\bar{z}\Vert^n}.
\]

In particular, if $n=1$,
\[
\psi_{1}^{(0,0)}=\frac 12 \frac y{|y|}.
\]

\paragraph{Embedding of real analytic forms\,:} Let $U$ and $V$ be as above.
We define a morphism
\[
\scA^{(p)}(U)\lra\scB^{(p)}(U)=H^{p,n}_{\bar\vt}(V,V\ssm U)
\]
by assigning to an element $\o(x)$ in $\scA^{(p)}(U)$ the class of $(\psi_{1}\wedge\o(z),\psi_{01}\wedge\o(z))$, where
$(\psi_{1},\psi_{01})$ is a representative of the  Thom class as above and  $\o(z)$ the complexification of  $\o(x)$. Note that $(\psi_{1}\wedge\o,\psi_{01}\wedge\o)$ is a cocycle as $\o$ is \h.  This induces an
embedding
$\iota^{(p)}:\scA^{(p)}(U)\hra\scB^{(p)}(U)$ compatible with the differentials $d$ of $\scA^{(\bullet)}(U)$ and $\scB^{(\bullet)}(U)$. In particular, if $p=0$, we have the embedding $\scA(U)\hra\scB(U)$, which is
given by $f\mapsto (\tilde f\psi_{1},\tilde f\psi_{01})$ with $\tilde f $ the complexification of $f$.

\subsection*{III. Some others}


We may develop the theory of Atiyah classes in the context of \v{C}ech-Dolbeault cohomology, which is
conveniently used to define
their localizations in the relative Dolbeault cohomology. In particular we have the $\bp$-Thom class
of a \h\ \vb, see \cite{ABST} and \cite{Su8} for details.

We refer to \cite{ASTT}
for a possible application of the above to the study of Hodge \str s under blowing-up.
We may equally use the complex of currents, instead
of that of differential forms, to define the relative Dolbeault cohomology. One of the advantages of this is that
the push-forward morphism is available, see \cite{Ta} for details and applications in the context of \cite{ASTT}.

\bibliographystyle{plain}

\vv

T. Suwa


Department of Mathematics 

Hokkaido University

Sapporo 060-0810, Japan

tsuwa@sci.hokudai.ac.jp
\end{document}